\documentclass[11pt]{article}
\usepackage{amsmath,amsthm,mathrsfs,bm}

\usepackage[initials,alphabetic]{amsrefs}
\usepackage[utf8]{inputenc}

\usepackage[dvipsnames]{xcolor}

\usepackage{amsfonts,amstext,amssymb,verbatim,epsfig,psfrag,stmaryrd,extarrows}
\usepackage{pgf,tikz}
\usepackage{float}
\usetikzlibrary{arrows}
\usepackage{hyperref}
\usepackage{color}
\usepackage{transparent}
\usepackage{subfigure}
\usepackage{prettyref}
\newrefformat{pr}{Proposition \ref{#1}}
\newrefformat{co}{Corollary \ref{#1}}
\newrefformat{le}{Lemma \ref{#1}}
\newrefformat{th}{Theorem \ref{#1}}
\newrefformat{ex}{Example \ref{#1}}
\newrefformat{re}{Remark \ref{#1}}
\newrefformat{se}{Section \ref{#1}}

\usepackage{esint}

\usepackage{bbm,mathrsfs}

\usepackage[normalem]{ulem}

\definecolor{forestgreen}{rgb}{0.1333,0.5451,0.1333}
\definecolor{navyblue}{rgb}{0,0,0.5}
\definecolor{darkgreen}{rgb}{0,0.3922,0}

\hypersetup{colorlinks=true, citecolor  = teal, linkcolor=MidnightBlue}
\makeatletter
\let\reftagform@=\tagform@
\def\tagform@#1{\maketag@@@{(\ignorespaces\textcolor{black}{#1}\unskip\@@italiccorr)}}
\renewcommand{\eqref}[1]{\textup{\reftagform@{\ref{#1}}}}
\makeatother

\makeatletter

\allowdisplaybreaks

\def\lessim{\ \lower4pt\hbox{$
		\buildrel{\displaystyle <}\over\sim$}\ }
\def\gessim{\ \lower4pt\hbox{$\buildrel{\displaystyle >}
		\over\sim$}\ }

\def\si{\sigma}

\newcommand{\pref}{\prettyref}

\newtheorem{lemma}{\bf Lemma}[section]

\newtheorem{theorem}[lemma]{\bf Theorem}

\newtheorem{example}[lemma]{\bf Example}

\newtheorem*{acknowledgements}{\bf Acknowledgements}
\theoremstyle{remark}

\newtheorem{remark}{Remark}[section]

\newtheorem{assumption}{\bf Assumption}

\numberwithin{equation}{section}

\newcommand{\8}{\infty}

\newcommand{\dx}{\mathcal{D}}

\newcommand{\nz}{\mathbb{N}}
\newcommand{\rz}{\mathbb{R}}

\newcommand{\ez}{\mathbb{E}}

\newcommand{\sfT}{\mathsf T}

\newcommand{\al}{\alpha}

\renewcommand{\si}{\sigma}

\newcommand{\la}{\lambda}

\newcommand{\Crt}{\mathrm{Crt}}
\newcommand{\Var}{\mathrm{Var}}
\newcommand{\Cov}{\mathrm{Cov}}
\newcommand{\GOE}{\mathrm{GOE}}

\newcommand{\dd}{\mathrm{d}}

\begin{document}

\title{On the expected number of critical points of \\ locally isotropic Gaussian random fields }
\author{Hao Xu\thanks{Department of Mathematics, University of Macau, yc17446@connect.um.edu.mo.} \and Haoran Yang\thanks{School of Mathematical Sciences, Peking University, yanghr@pku.edu.cn.} \and Qiang Zeng\thanks{Department of Mathematics, University of Macau, qzeng.math@gmail.com, research partially supported by SRG 2020-00029-FST and FDCT 0132/2020/A3.}}

\maketitle

\abstract
We consider locally isotropic Gaussian random fields on the $N$-dimensional Euclidean space for fixed $N$. Using the so called Gaussian Orthogonally Invariant matrices first studied by Mallows in 1961 which include the celebrated Gaussian Orthogonal Ensemble (GOE), we establish the Kac--Rice representation of expected number of critical points of non-isotropic Gaussian fields, complementing the isotropic case obtained by Cheng and Schwartzman in 2018. In the limit $N=\8$, we show that such a representation can be always given by GOE matrices, as conjectured by Auffinger and Zeng in 2020.

\section{Introduction}

Locally isotropic random fields on the $N$-dimensional Euclidean space $\rz^N$ were introduced by Kolmogorov in 1941 \cite{Ko41} for the application in statistical theory of turbulence. Since then, this class of stochastic processes has been extensively studied in both physics and mathematics. In particular, locally isotropic Gaussian random fields have been used to model a particle confined in a random potential and serve as a toy model for the elastic manifold. For an incomplete list of literature, we refer the interested reader to \cite{MP91,En93,Fy04,FS07,FB08,FN12,FLD18,FLD20} for the background in physics and \cite{Kli12,AZ20,AZ22,BBMsd,BBMsd2,XZ22} for the mathematical development.

The goal of this paper is two-fold. For application to statistical physics, frequently we need to send $N\to\8$ for the thermodynamic limit, which puts additional restrictions on the class of such Gaussian fields. Using the Kac–Rice formula and the connection to random matrices, in the seminal work \cite{Fy04} Fyodorov considered the expected number of critical points of isotropic Gaussian random fields on $\rz^N$ in the asymptotic regime $N\to\8$. This is commonly known as landscape complexity (or complexity for simplicity) in statistical physics; see also \cite{FN12,Fy15} for related topics. Recently, Auffinger and Zeng provided a detailed study on the complexity of non-isotropic Gaussian random fields with isotropic increments \cite{AZ20,AZ22}. On the other hand, for applications to statistics and other fields, it is also of interest to consider random fields on a finite-dimensional Euclidean space. Guided by this principle, Cheng and Schwartzman gave a representation of the expected number of critical points of isotropic Gaussian random fields on $\rz^N$ with fixed $N$ {\cite[Theorem 3.5]{CS18}}. Here an apparent gap arises: What about the non-isotropic Gaussian fields on the fixed $\rz^N$? We provide several answers to this question. Furthermore, we show that a technical assumption used in \cite{AZ20} is redundant in the large $N$ limit as conjectured by the authors while it is useful to obtain a representation via matrices from the Gaussian Orthogonal Ensemble (GOE) in finite dimensions.

Let us be more precise. A locally isotropic Gaussian random field  $H_N=\left\{H_N(x): x \in \mathbb{R}^N\right\}$ is a centered Gaussian process indexed by $\rz^N$ that satisfies
\begin{align}\label{eq:def}
    \mathbb{E}\left[\left(H_{N}(x)-H_{N}(y)\right)^{2}\right]= D_N\left(\|x-y\|^{2}\right), \quad x, y \in \mathbb{R}^{N}.
\end{align}
Here the function $D_N \colon \mathbb{R}_{+} \to \mathbb{R}_{+}$ is called the structure function of $H_N$, $\|\cdot\|$ is the Euclidean norm and $\mathbb{R}_{+}=[0, \infty)$. The process $H_N$ is also known as a Gaussian random field with isotropic increments. The condition \eqref{eq:def} determines the law of $H_{N}$ up to an additive shift by a Gaussian random variable. Following \cite{Ya87}, we recall some basic properties of the structure function. Let $\mathcal{D}_N$ denote the set of all $N$-dimensional structure functions and   $\mathcal{D}_{\infty}$ the set of structure functions which belong to $\mathcal{D}_N$ for all $N\in\nz$. Since any $N$-dimensional structure function is necessarily an $M$-dimensional structure function for all integers $M$ less than $N$, it is clear that
$$
\mathcal{D}_1 \supset \mathcal{D}_2 \supset \dots \supset \mathcal{D}_N \supset \dots \supset \mathcal{D}_{\infty},
$$
where the symbol $\supset$ denotes inclusion. In the following, we write $D_\8\in\mathcal{D}_{\infty}$ for the structure function of a field that can be defined on $\rz^N$ for all natural numbers $N\in\nz$ and in this case we frequently write $N=\8$.
Let us define
\begin{align*}
\Lambda_N(x) =
\begin{cases}
    \cos x,& N=1,\\
    2^{(N-2) / 2} \Gamma\left(\frac{N}{2}\right) \frac{J_{(N-2) / 2}(x)}{x^{(N-2) / 2}} =1-\frac{x^2}{2 N}+\frac{x^4}{2 \cdot 4 \cdot N(N+2)}-\cdots,& N=2,3, \dots, \\
    e^{-x^2}, & N=\infty.
\end{cases}
\end{align*}
Here $J_N$ is the $N$th Bessel function of the first kind, which has the following representation
$$
J_N(x)=\sum_{m=0}^{\infty} \frac{(-1)^m}{m ! \Gamma(m+N+1)}\left(\frac{x}{2}\right)^{2 m+N}.
$$
From here it can be shown that $\Lambda_N(\sqrt{2N}x)\stackrel{N \rightarrow \infty}{\longrightarrow} e^{-x^2}$.
By \cite{Sch38,Ya57}, locally isotropic Gaussian random fields (or equivalently the class $\dx_N$) can be classified into two cases:

\begin{enumerate}
    \item Isotropic fields. There exists a function $B_N \colon \mathbb{R}_{+} \to \mathbb{R}$ such that
\begin{align}\label{Hcov}
\mathbb{E}\left[H_{N}(x) H_{N}(y)\right]= B_N\left(\|x-y\|^{2}\right),
\end{align}
where $B_N$ has the representation
\begin{align*}
B_N(r)=C_{N}+\int_{(0, \infty)} \Lambda_N(\sqrt{rt}) \nu_N(\mathrm{d} t),
\end{align*}
with a constant $C_{N} \in \mathbb{R}_{+}$ and a finite measure $\nu_N$ on $(0, \infty)$. Clearly, in this case we have $D_N(r)=2(B_N(0)-B_N(r))$. In particular, for the class $\dx_\8$, we write $B(r)=B_\8(r)$ and it can be represented as
$$
B(r) = C + \int_{(0, \infty)} e^{-r t} \nu (\mathrm{d} t).
$$

\item Non-isotropic fields with isotropic increments. The structure function $D_N$ can be written as
\begin{align} \label{DNr}
D_N(r) = \int_{(0, \infty)} \left( 1 - \Lambda_N ( \sqrt{rt} ) \right) \nu_N (\mathrm{d} t) + A_N r,
\end{align}
where $A_N \in \mathbb{R}_{+}$ is a constant and $\nu_N$ is a $\sigma$-finite measure with
$$
\int_{(0, \infty)} \frac{t}{1+t} \nu_N(\mathrm{d} t)<\infty.
$$
For $N=\infty$, the structure function $D(r):=D_\8(r)$
has the following form
\begin{align}\label{Dr}
D(r)=\int_{(0, \infty)}\left(1-e^{-rt}\right) \nu(\mathrm{d} t)+A r,
\end{align}
which is, in the language of Theorem \ref{Bern} below, a Bernstein function with $D(0)=0$.
\end{enumerate}

These representations provide us some information on the sign of derivatives provided they are finite.
For instance, \eqref{Dr} implies that $D'(r)\geq 0$ and $D''(r)\leq 0$ for $r\geq 0$. However, due to the oscillatory nature of Bessel functions, we only have $D_N'(0)\geq 0$ and $D_N''(0)\leq 0$ for $N\in\nz$.  

Let $E \subset \mathbb{R}$ and $T_N\subset\mathbb{R}^N$ be Borel subsets. If the random field is twice differentiable, we define
\begin{align*}
\operatorname{Crt}_N\left(E, T_N\right) & =\#\left\{x \in T_N: \nabla H_N(x)=0,  H_N(x) \in E\right\}, \\
\operatorname{Crt}_{N, k}\left(E, T_N\right) & =\#\left\{x \in T_N: \nabla H_N(x)=0, \,  H_N(x) \in E, \, i\left(\nabla^2 H_N(x)\right)=k\right\}, \quad k = 0, 1, \dots, N,
\end{align*}
where $i (\nabla^2 H_N(x) )$ is the index (or the number of negative eigenvalues) of the Hessian $\nabla^2 H_N(x)$. Under some suitable smoothness conditions in \cite{AT07}, Cheng and Schwartzman gave a representation of $\ez[\Crt_{N,k}([u,\8),T)]$ for the \emph{isotropic fields} using the Kac--Rice formula in {\cite[Theorem 3.5]{CS18}}, where $u\in[-\8,\8)$ and $T$ is the unit-volume ball in $\mathbb{R}^N$. A key ingredient in the representation is what the authors call the Gaussian Orthogonally Invariant (GOI) matrix, which was first studied by Mallows \cite{Ma61}. Following their terminology, an $N \times N$ real symmetric random matrix $M$ is said to be Gaussian Orthogonally Invariant (GOI) with covariance parameter $c$, denoted by $\operatorname{GOI}(c)$, if its entries $M_{i j}, 1\leq i,j\leq N$ are centered Gaussian random variables such that
$$
\mathbb{E}\left[M_{i j} M_{k l}\right]=\frac{1}{2}\left(\delta_{i k} \delta_{j l}+\delta_{i l} \delta_{j k}\right)+c \delta_{i j} \delta_{k l}.
$$
Note that the distribution of such matrices is invariant under the congruence transformation by any orthogonal matrix. Therefore, they share some properties of the GOE matrices.
In particular, for $c=0$, GOI$(0)$ matrices are exactly GOE matrices. Recall that a Gaussian vector is nondegenerate if and only if the covariance matrix of the Gaussian vector is positive definite. Cheng and Schwartzman \cite{CS18} showed that a GOI$(c)$ matrix of size $N\times N$ is nondegenerate if and only if $c>-\frac{1}{N}$. Moreover, for this nondegenerate GOI$(c)$ matrix, they derived that the density of the ordered eigenvalues is given by
\begin{align}\label{eq:goidens}
f_c\left(\lambda_1, \ldots, \lambda_N\right)= & \frac{1}{K_N \sqrt{1+N c}} \exp \left\{-\frac{1}{2} \sum_{i=1}^N \lambda_i^2+\frac{c}{2(1+N c)}\left(\sum_{i=1}^N \lambda_i\right)^2\right\}\nonumber \\
& \times \prod_{1 \leq i<j \leq N}\left|\lambda_i-\lambda_j\right| \mathbf{1}{\left\{\lambda_1 \leq \cdots \leq \lambda_N\right\}},
\end{align}
where $K_N=2^{N / 2} \prod_{i=1}^N \Gamma\left(\frac{i}{2}\right)$ is the normalization constant. For $c\neq0$, one may write a GOI$(c)$ matrix as a GOE matrix plus a random scalar matrix. If $c>0$ the scalar matrix is independent of the GOE matrix,  while if $c<0$ the scalar matrix is no longer independent. In the limit $N\to\8$ or equivalently for structure functions $D\in\dx_\8$, the covariance parameter $c$ is positive and thus for these structure functions and all dimensions $N\in\nz$, the Kac--Rice representation of $\ez[\Crt_{N,k}(E,T_N)]$ can always be given by GOE matrices. This was the setting in the pioneering works of Fyodorov and Nadal \cite{Fy04, FN12} where the GOE matrices are crucial for asymptotic analysis. However, if one insists on considering structure functions from $\dx_N$ for fixed $N$, then additional restriction is needed in order to get a GOE matrix (plus an independent scalar matrix) in the Kac--Rice representation.

Before illustrating similar problems for the non-isotropic Gaussian fields, let us first put a remark on the relationship between the two cases. In general, the two types of Gaussian random fields are quite different, like the Ornstein--Uhlenbeck process and Brownian motion in dimension $1$. But if $\nu_N$ is a finite measure in the second case, the non-isotropic field is essentially a shifted isotropic field. Indeed, let $H_N(x)$ be an isotropic field that satisfies  \begin{equation*}
	\mathbb{E}\left[H_{N}(x) H_{N}(y)\right]=\frac{1}{2}\int_{(0, \infty)} \Lambda_N \left( \sqrt{\|x-y\|^{2}t} \right) \nu_N(\mathrm{d} t).
\end{equation*}
Then we can verify that $\widehat{H}_N(x):=H_N(x)-H_N(0)$ satisfies
\begin{equation}\label{eq:hnxy}
	\mathbb{E}\left[\left(\widehat{H}_{N}(x)-\widehat{H}_{N}(y)\right)^{2}\right]= \int_{(0, \infty)}\left( 1 - \Lambda_N \left( \sqrt{\|x-y\|^{2}t} \right) \right) \nu_N(\mathrm{d} t),
\end{equation}
and $\widehat{H}_N(0)=0$, which means that $\widehat{H}_N(x)$ is a non-isotropic Gaussian  random field with isotropic increments. Conversely, let $\widehat{H}_N(x)$ be a non-isotropic field satisfying \eqref{eq:hnxy} with some finite measure $\nu_N$. Define
$H_N(x)=\widehat{H}_N(x)+H_N(0)$, where
$H_N(0)$ is a centered Gaussian random variable with variance $\ez[H_N(0)^2]=\frac{1}{2}|\nu_N|$ and
$$\mathbb{E}\left[\widehat{H}_{N}(x) H_{N}(0)\right]=\frac{1}{2}\int_{(0, \infty)} \left( \Lambda_N \left( \sqrt{\|x\|^{2}t} \right) - 1 \right) \nu_N(\mathrm{d} t).$$
One can check that $H_N(x)$
is an isotropic Gaussian random field. Due to the relation $H_N(x)-H_N(0)=\widehat{H}_N(x)$, the non-isotropic field $\widehat{H}_N (x)$ has the same critical points and landscape as those of the isotropic field $H_N$.
Furthermore, let
$\widetilde{H}_N(x)$ be a non-isotropic Gaussian random field satisfying
\begin{equation*}
\mathbb{E}\left[\left(\widetilde{H}_{N}(x)-\widetilde{H}_{N}(y)\right)^{2}\right]= \int_{(0, \infty)}\left( 1 - \Lambda_N \left( \sqrt{\|x-y\|^{2}t} \right) \right) \nu_N(\mathrm{d} t)+A_N \|x-y\|^{2},
\end{equation*}
with some finite measure $\nu_N$ and constant $A_N >0$. In this case, we can split $\widetilde{H}_{N}(x)$ into two independent non-isotropic parts $\widetilde{H}_{N}(x)=\widehat{H}_{N}(x)+\overline{H}_{N}(x)$, where
$\widehat{H}_{N} (x)$ satisfies
\eqref{eq:hnxy} and $\overline{H}_{N}(x)$ satisfies
\begin{equation*}
\mathbb{E}\left[\left(\overline{H}_{N}(x)-\overline{H}_{N}(y)\right)^{2}\right]=A_N \|x-y\|^{2}.
\end{equation*}
Note that we may write $\overline{H}_{N}(x)= \langle x,\xi\rangle$ for a centered Gaussian random vector $\xi=(\xi_1,\xi_2,\dots,\xi_N)$ with covariance matrix $A_N\mathbf{I}_N$, where $\langle \cdot,\cdot\rangle$ is the Euclidean inner product and $\mathbf{I}_{N}$ denotes the $N\times N$ identity matrix hereafter. In this case, the field $\overline{H}_N$ is almost trivial since its gradient is a fixed random vector and we have $\nabla\widetilde{H}_{N}(x)=0$ if and only if  $\nabla \widehat{H}_{N}(x)=-\xi$. Therefore, when we talk about the non-isotropic fields, we may assume that $\lim_{r\to\8} D(r)=\8$ and the measure $\nu_N$ is $\sigma$-finite but not finite.

For the non-isotropic fields with isotropic increments, it is expected that the above representation problem is more challenging since the distribution of the field $H_N(x)$ depends on the location variable $x\in \rz^N$.
Recently, Auffinger and Zeng considered the Kac--Rice representation of $\ez[\operatorname{Crt}_{N, k}\left(E, T_N\right)]$ for the non-isotropic Gaussian fields with isotropic increments in the limit $N\to\8$ during their study of the landscape complexity in \cite{AZ20,AZ22}, where GOE matrices are indispensable in the analysis as the isotropic case. This left the representation problem open for the fixed finite dimensions. Moreover, in order to utilize GOE matrices, they assumed a technical condition (see Assumption \ref{assumption3} below) and verified it for some special cases or subclass of $\dx_\8$. They conjectured that this condition always holds for all structure functions belonging to $\dx_\8$.

This paper aims to resolve these issues for the non-isotropic Gaussian fields with isotropic increments. In Section \ref{sec2}, we prove the main results of this paper, which display various representations of $\ez[\operatorname{Crt}_{N, k}\left(E, T_N\right)]$. The basic tool is the Kac--Rice formula as usual. In general, the representation can be obtained by using GOI$(c)$ matrices (see Theorem \ref{Shell}). 
Furthermore, we may reduce the GOI$(c)$ matrices to GOE matrices in the representation, in the special case of $E=\rz$ (see Theorem \ref{ER}) or with Assumption \ref{assumption3} (see Theorem \ref{shell}). 
These results can be regarded as the non-isotropic analog of those for the isotropic Gaussian fields in \cite{CS18}. In Section \ref{sec3}, we show that for the structure functions in $\dx_\8$, $\ez [ \operatorname{Crt}_{N, k} ( E, T_N ) ]$ can always be represented by using GOE matrices when $T_N$ is a shell domain. In other words, the aforementioned technical condition is redundant as conjectured in \cite{AZ20}.

\section{Representations for non-isotropic Gaussian fields on $\rz^N$} \label{sec2}

\subsection{A perturbed GOI matrix model}

For clarity of exposition, we introduce a random matrix model which is a special perturbation of the GOI model. We call an $N\times N$ real symmetric random matrix $M$ Spiked Gaussian Orthogonally Invariant (SGOI) with parameters $d_1,d_2$ and $d_3$, denoted by $\operatorname{SGOI}(d_1,d_2,d_3)$, if its entries $M_{i j}, 1\leq i,j\leq N$ are centered Gaussian random variables such that
\begin{align}\label{copgoi}
\mathbb{E}\left[M_{i j} M_{k l}\right]=\frac{1}{2}\left(\delta_{i k} \delta_{j l}+\delta_{i l} \delta_{j k}\right)+d_1 \delta_{i j} \delta_{k l}+d_2(\delta_{i 1}\delta_{ j1}\delta_{k l}+\delta_{k1}\delta_{l1}\delta_{i j})+d_3\delta_{i 1}\delta_{ j1}\delta_{k1}\delta_{l1}.
\end{align}
Clearly, if $d_2=0$ and $d_3=0$, the $\operatorname{SGOI}(d_1,0,0)$ matrix is exactly a $\operatorname{GOI}(c)$ matrix with $c=d_1$.
\begin{lemma}\label{NonPGOI}
Let $M$ be an $N\times N$ $\operatorname{SGOI}(d_1,d_2,d_3)$ matrix. Then $M$ is nondegenerate if and only if $d_1>-\frac{1}{N-1}$ and $1+d_3+\frac{d_1+2d_2-(N-1)d_2^2}{1+(N-1)d_1}>0$.
\end{lemma}
\begin{proof}
Since $M$ is symmetric and the diagonal entries are independent of the off-diagonal entries, $M$ is nondegenerate if and only if the vector of diagonal entries $(M_{11},M_{22},\dots,M_{NN})$ is nondegenerate. By \eqref{copgoi}, the covariance matrix of the vector $(M_{11},M_{22},\dots,M_{NN})$ is
\begin{align}\label{eq:covM}
	\Theta =
    \begin{pmatrix}
      1+d_1+2d_2+d_3& d_1+d_2
      & \dots & d_1+d_2 \\
       d_1+d_2  &  d_1+1 & \dots & d_1 \\
      \vdots & \vdots &\ddots  & \vdots \\d_1+d_2 & d_1& \dots &  d_1+1 \\
    \end{pmatrix}.
  \end{align}
  It is straightforward to check that the lower right $(N-1)\times (N-1)$ submatrix of $\Theta$ is positive definite if and only if $d_1>-\frac{1}{N-1}$. Note that $\Theta$ is positive definite if and only if the determinant of $\Theta$ is larger than $0$ and the lower right $(N-1)\times (N-1)$ submatrix of $\Theta$ is positive definite. By the Schur complement formula, one can show that
  \begin{equation*}
      \det \Theta = [1+(N-1)d_1] \left[1+d_3+\frac{d_1+2d_2-(N-1)d_2^2}{1+(N-1)d_1}\right],
  \end{equation*}
  which gives the desired result.
\end{proof}

Let $M$ be an $N\times N$ $\operatorname{SGOI} (d_1,d_2,d_3)$ matrix. In general, $M$ can be represented as the following form
\begin{align}\label{M}
   M =\begin{pmatrix}
      \zeta_1& \xi^\mathsf T \\
       \xi & \operatorname{GOI}(d_1)
    \end{pmatrix}=
    \begin{pmatrix}
      \zeta_1& \xi^\mathsf T \\
       \xi & \GOE_{N-1}+\zeta_2\mathbf{I}_{N-1}
    \end{pmatrix} ,
\end{align}
where $\xi$ is a centered column Gaussian vector with covariance matrix $\frac{1}{2}\mathbf{I}_{N-1}$, which is independent of $\zeta_1,\zeta_2$ and the $(N-1) \times (N-1)$ GOE matrix $\GOE_{N-1}$. To represent $M$ as explicitly as possible, we need to determine the relationship among $\zeta_1,\zeta_2$ and the diagonal elements of $\GOE_{N-1}$. Since for $i=1,\dots,N-1$, 
$$1+d_1=\Var[(\GOE_{N-1})_{ii}+\zeta_2]=1+2\Cov[(\GOE_{N-1})_{ii}, \zeta_2]+\Var(\zeta_2),$$ where $(\GOE_{N-1})_{ii}$'s are the diagonal elements of $\GOE_{N-1}$, we find that $\Cov[(\GOE_{N-1})_{ii}, \zeta_2]$ is a constant for $1\le i\le N-1$. Let $\varsigma=\Cov(\zeta_1,\zeta_2)$ and $\vartheta=\Cov(\zeta_2,(\GOE_{N-1})_{ii})$. The covariance matrix of $(\zeta_1,\zeta_2,(\GOE_{N-1})_{ii, 1\le i\le N-1})$ is
\begin{align}\label{eq:Xi}
    \Xi=\begin{pmatrix}
        1+d_1+2d_2+d_3& \varsigma& d_1+d_2-\varsigma &\dots& d_1+d_2-\varsigma \\
        \varsigma & d_1-2\vartheta & \vartheta&\dots &\vartheta\\
        d_1+d_2-\varsigma & \vartheta &1& \dots& 0\\
        \vdots&\vdots &\vdots & \ddots &\vdots\\
        d_1+d_2-\varsigma &\vartheta &0 &\dots&1
    \end{pmatrix}.
\end{align}
A moment of reflection shows that $M$ is nondegenerate if and only if $\Xi$ is positive definite. Therefore, we just need to choose $\varsigma$ and $\vartheta$ such that
\begin{align*}
    \frac{-1-\sqrt{1+d_1(N-1)}}{N-1}< \vartheta <\frac{-1+\sqrt{1+d_1(N-1)}}{N-1}\quad \text{and}\quad \det \Xi>0.
\end{align*}
If $d_1\ge0$ and we set $\vartheta=0$ as in \cite[(2.2)]{CS18}, we find a sufficient condition for $\varsigma$:
\begin{align}\label{eq:d1pos}
    1+d_1+2d_2+d_3-\frac{\varsigma^2}{d_1}-(N-1)(d_1+d_2-\varsigma)^2>0.
\end{align}
If $d_1<0$ and we set $\vartheta=d_1$ as in \cite{CS18}, we will get a more involved condition. For concrete examples, if $d_1\ge0$, we may take $\vartheta=0$ and $\varsigma=\frac{d_1^2+d_1d_2}{1+d_1}$; if $d_1<0$, we may set $\vartheta=d_1$ and $\varsigma=0$.


Note that the distribution of SGOI matrices is not invariant under the orthogonal congruence transformation, which makes it hard to deduce the joint eigenvalue density for this matrix model. Fortunately, it is sufficient to use the joint eigenvalue density of GOI matrices \pref{eq:goidens} for our representation formulas below.
By the conditional distribution of Gaussian vectors,  we have \begin{align}\label{goicond}
    \mathbb{E}(\GOE_{N-1}+\zeta_2\mathbf{I}_{N-1}\mid \zeta_1=y)=&\frac{(d_1+d_2)y}{1+d_1+2d_2+d_3}\mathbf{I}_{N-1},\nonumber\\
   \left( \GOE_{N-1}+\zeta_2\mathbf{I}_{N-1}\mid \zeta_1=y\right)\stackrel{d}{=}&\frac{(d_1+d_2)y}{1+d_1+2d_2+d_3}\mathbf{I}_{N-1}+M',
\end{align}  where $M'$ is an $(N-1)\times (N-1)$ GOI$(c)$ matrix with parameter $c=\frac{d_1+d_1d_3-d_2^2}{1+d_1+2d_2+d_3}$. Since we always have
$\mathbb{E}[\zeta_1(\GOE_{N-1}+\zeta_2\mathbf{I}_{N-1})]=(d_1+d_2)\mathbf{I}_{N-1},
$ the conditional distribution \eqref{goicond} is invariant for different constructions of SGOI matrices. This is crucial in Theorem \ref{Shell} below, and the freedom to specify $\varsigma$ and $\vartheta$ in \pref{eq:Xi} allows us to choose matrix models that are easier to work with in Theorem \ref{shell}. In the following, we will only need the covariances given in \pref{eq:covM} but not the concrete relationship among $\zeta_1,\zeta_2$ and $\GOE_{N-1}$ until deriving Theorem \ref{shell}.

\subsection{The nondegeneracy condition}
According to \cite{AZ20}, we need the following  assumption for the model so that the random fields are twice differentiable  almost surely.

\begin{assumption}[Smoothness] \label{assumption1}
    For fixed $N\in\nz\cup\{\infty\}$, the function $D_N$ is four times differentiable at 0 and it satisfies
    \begin{equation*}
    0 < \left|D_N^{(4)}(0)\right| < \infty .
    \end{equation*}
\end{assumption}
For the non-isotropic Gaussian random fields with isotropic increments, we also need the following assumption, which is a natural assumption for stochastic processes with stationary increments.

\begin{assumption}[Pinning] \label{assumption2}
    We have $H_{N}(0)=0.$     
\end{assumption}
Let us recall the covariance structure of the non-isotropic Gaussian fields with isotropic increments.
\begin{lemma}[{\cite[Lemma A.1]{AZ20}}]
\label{Corv}
 Assume Assumptions \ref{assumption1} and \ref{assumption2}. Then for $x \in \mathbb{R}^{N}$,
\begin{align*}
\operatorname{Cov}\left[H_{N}(x), H_{N}(x)\right] &=D_N\left(\|x\|^{2}\right), \\
\operatorname{Cov}\left[H_{N}(x), \partial_{i} H_{N}(x)\right] &=D_N^{\prime}\left(\|x\|^{2}\right) x_{i}, \\
\operatorname{Cov}\left[\partial_{i} H_{N}(x), \partial_{j} H_{N}(x)\right] &=D_N^{\prime}(0) \delta_{i j}, \\
\operatorname{Cov}\left[H_{N}(x), \partial_{i j} H_{N}(x)\right] &=2 D_N^{\prime \prime}\left(\|x\|^{2}\right) x_{i} x_{j}+\left[D_N^{\prime}\left(\|x\|^{2}\right)-D_N^{\prime}(0)\right] \delta_{i j}, \\
\operatorname{Cov}\left[\partial_{k} H_{N}(x), \partial_{i j} H_{N}(x)\right] &=0, \\
\operatorname{Cov}\left[\partial_{l k} H_{N}(x), \partial_{i j} H_{N}(x)\right] &=-2 D_N^{\prime \prime}(0)\left[\delta_{j l} \delta_{i k}+\delta_{i l} \delta_{k j}+\delta_{k l} \delta_{i j}\right],
\end{align*}
where $\delta_{i j}$ denotes the Kronecker delta function, $i, j, k, l = 1, \dots, N$.
\end{lemma}
Cheng and Schwartzman {\cite[Proposition 3.3]{CS18}} studied the nondegeneracy condition of the isotropic Gaussian random fields, and they showed that the Gaussian vector $(H_N(x), \nabla H_N(x), \partial_{i j} H_{N}(x),1\leq i\leq j\leq N)$ is nondegenerate if and only if  $\frac{B_N^{\prime}(0)^2}{B_N^{\prime\prime}(0)}< \frac{N+2}{N}.$  The following lemma gives the nondegeneracy condition of the non-isotropic Gaussian random fields with isotropic increments.

\begin{lemma} \label{Nond}
Assume Assumptions \ref{assumption1} and \ref{assumption2} hold for the structure function $D_N$ and the associated Gaussian field $H_N$. Then for any fixed $x \in \mathbb{R}^N \setminus \{0\}$,
 the Gaussian vector $(H_N(x), \nabla H_N(x), \partial_{i j} H_{N}(x),1\leq i\leq j\leq N)$ is nondegenerate if and only if for $r=\|x\|^{2}$,
\begin{equation}\label{nondg}
D_N\left(r\right)-\frac{D_N^{\prime}(r)^2 r}{D_N^{\prime}(0)}
+\frac{(N+1)D_N^{\prime\prime}(r)^2 r^{2}}{(N+2)D_N^{\prime\prime}(0)} +\frac{2rD_N^{\prime\prime}(r) ( D_N^{\prime}(r)-D_N^{\prime}(0) )}{(N+2)D_N^{\prime\prime}(0)}
+\frac{N ( D_N^{\prime}(r)-D_N^{\prime}(0) )^2}{2(N+2)D_N^{\prime\prime}(0)}>0.
\end{equation}
Moreover, if $D\in \dx_N$ for some $N\in\nz\cup\{\8\}$, Assumptions \ref{assumption1} and \ref{assumption2} hold, and we have for $r=\|x\|^{2}$, 
\begin{align}\label{larn}
D\left(r\right)-\frac{D^{\prime}(r)^2 r}{D^{\prime}(0)}
+\frac{D^{\prime\prime}(r)^2 r^{2}}{D^{\prime\prime}(0)}
+\frac{\left(D^{\prime}(r)-D^{\prime}(0)\right)^2}{2D^{\prime\prime}(0)}>0,
\end{align}
then the Gaussian vector $(H_N(x), \nabla H_N(x), \partial_{i j} H_{N}(x),1\leq i\leq j\leq N)$ is nondegenerate. Here if $N=\8$, we understand that the indices $i,j\in \nz$.
\end{lemma}
\begin{proof}
Following Lemma \ref{Corv}, it is clear that the covariance matrix of the Gaussian vector 
\begin{equation*}
	\left(H_N(x), \nabla H_N(x), \partial_{i j} H_{N}(x),1\leq i\leq j\leq N\right),
\end{equation*}
is given by
\begin{align}
    \textbf{$\mathbf{G}$}=
    \begin{pmatrix}
      D_N\left(\|x\|^{2}\right)& \xi_1^\mathsf T& \xi_2^\mathsf T& \xi_3^\mathsf T \\
       \xi_1  &  \textbf{$\mathbf{A}$} & \mathbf{0}_1^\mathsf T & \mathbf{0}_2^\mathsf T \\
       \xi_2  & \mathbf{0}_1 &  \textbf{$\mathbf{B}$}& \mathbf{0}_3^\mathsf T \\
       \xi_3  & \mathbf{0}_2 & \mathbf{0}_3 &  \textbf{$\mathbf{C}$} \\
    \end{pmatrix} \nonumber ,
  \end{align}
where
$$\xi_1=\left(D_N^{\prime}(\|x\|^{2})x_1,\dots,D_N^{\prime}(\|x\|^{2})x_N \right)^\mathsf T,$$ is the covariance vector of $H_{N}(x)$ with $\left(\partial_{1} H_{N}(x),\dots,\partial_{N} H_{N}(x) \right)$,
$$\xi_2=\left(2D_N^{\prime\prime}(\|x\|^{2})x_1^2+D_N^{\prime}(\|x\|^{2})-D_N^{\prime}(0), \dots,
2D_N^{\prime\prime}(\|x\|^{2})x_N^2+D_N^{\prime}(\|x\|^{2})-D_N^{\prime}(0) \right)^\mathsf T,$$ is the covariance vector of $H_{N}(x)$ with $\left(\partial_{11} H_{N}(x),\dots,\partial_{NN} H_{N}(x) \right)$,
$$\xi_3=2D_N^{\prime\prime}(\|x\|^{2})\left(x_1x_2,\dots,x_1x_N,
x_2x_3,\dots,x_2x_N,\dots, x_{N-1}x_N \right)^\mathsf T,$$ is an $\frac{N(N-1)}{2}$-dimensional vector, which is the covariance vector of $H_{N}(x)$ with the Hessian above the diagonal, $\mathbf{0}_1$, $\mathbf{0}_2$ and $\mathbf{0}_3$ are $N\times N$, $\frac{N(N-1)}{2}\times N$, and $\frac{N(N-1)}{2}\times N$  matrices with all elements equal to zero, respectively. In addition to this,
\begin{equation*}
	\mathbf{A}=D_N^{\prime}(0)\mathbf{I}_N, \qquad
	\mathbf{B}=-4 D_N^{\prime \prime}(0)\mathbf{I}_N-2 D_N^{\prime \prime}(0)\mathbf{1}_N \mathbf{1}_N^\mathsf T, \qquad \mathbf{C}=-2D_N^{\prime\prime}(0)\mathbf{I}_{\frac{N(N-1)}{2}},
\end{equation*}
where $\mathbf{1}_N$ is the $N$-dimensional column vector with all elements equal to one hereafter.
It is clear that all of $\mathbf{A},\mathbf{B}$ and $\mathbf{C}$ are positive definite matrices. By the eigenvalue interlacing theorem, the matrix $\mathbf{G}$ has at most one negative eigenvalue. Therefore, 
the positive definiteness of the matrix $\mathbf{G}$ is equivalent to $\det \mathbf{G}>0$. Noting that $\mathbf{B}^{-1}=\frac{-1}{4(N+2)D_N^{\prime\prime}(0)}\left[(N+2)\mathbf{I}_N-\mathbf{1}_N \mathbf{1}_N^T\right]$, we may consider the Schur complement of the block diagonal submatrix $\mathrm{Diag}(\mathbf{A},\mathbf{B},\mathbf{C})$ in $\mathbf{G}$ and find
\begin{multline*}
\frac{\det \mathbf{G}}{\det \mathbf{A}\det \mathbf{B}\det \mathbf{C}}=D_N\left(\|x\|^{2}\right)-\frac{D_N^{\prime}(\|x\|^{2})^2\|x\|^{2}}{D_N^{\prime}(0)}
+\frac{(N+1)D_N^{\prime\prime}(\|x\|^{2})^2\|x\|^{4}}{(N+2)D_N^{\prime\prime}(0)} \\
+\frac{2D_N^{\prime\prime}(\|x\|^{2})\|x\|^{2}\left(D_N^{\prime}(\|x\|^{2})-D_N^{\prime}(0)\right)}{(N+2)D_N^{\prime\prime}(0)}
+\frac{N\left(D_N^{\prime}(\|x\|^{2})-D_N^{\prime}(0)\right)^2}{2(N+2)D_N^{\prime\prime}(0)}.
\end{multline*}
The proof of the first claim is completed by replacing $\|x\|^{2}$ with $r$.

For the second assertion, we need to show that
\begin{equation*}
D\left(r\right)-\frac{D^{\prime}(r)^2 r}{D^{\prime}(0)}
+\frac{(N+1)D^{\prime\prime}(r)^2 r^{2}}{(N+2)D^{\prime\prime}(0)} +\frac{2rD^{\prime\prime}(r)\left(D^{\prime}(r)-D^{\prime}(0)\right)}{(N+2)D^{\prime\prime}(0)}
+\frac{N\left(D^{\prime}(r)-D^{\prime}(0)\right)^2}{2(N+2)D^{\prime\prime}(0)}>0.
\end{equation*}
Notice that
\begin{align*}
&\mathrel{\phantom{=}} D\left(r\right)-\frac{D^{\prime}(r)^2 r}{D^{\prime}(0)}
+\frac{(N+1)D^{\prime\prime}(r)^2 r^{2}}{(N+2)D^{\prime\prime}(0)}
+\frac{2rD^{\prime\prime}(r)\left(D^{\prime}(r)-D^{\prime}(0)\right)}{(N+2)D^{\prime\prime}(0)}
+\frac{N\left(D^{\prime}(r)-D^{\prime}(0)\right)^2}{2(N+2)D^{\prime\prime}(0)}\nonumber\\
&= D\left(r\right)-\frac{D^{\prime}(r)^2 r}{D^{\prime}(0)}
+\frac{D^{\prime\prime}(r)^2 r^{2}}{D^{\prime\prime}(0)}
+\frac{\left(D^{\prime}(r)-D^{\prime}(0)\right)^2}{2D^{\prime\prime}(0)}
\nonumber\\
&\quad + \frac{\left(D^{\prime}(r)-D^{\prime}(0)\right)^2}{-(N+2)D^{\prime\prime}(0)}+\frac{D^{\prime\prime}(r)^2 r^{2}}{-(N+2)D^{\prime\prime}(0)}-\frac{2rD^{\prime\prime}(r)\left(D^{\prime}(r)-D^{\prime}(0)\right)}{-(N+2)D^{\prime\prime}(0)}.
\end{align*}
By the elementary inequality $a^2+b^2\geq 2ab$, we have
$$\frac{\left(D^{\prime}(r)-D^{\prime}(0)\right)^2}{-(N+2)D^{\prime\prime}(0)}+\frac{D^{\prime\prime}(r)^2 r^{2}}{-(N+2)D^{\prime\prime}(0)}-\frac{2rD^{\prime\prime}(r)\left(D^{\prime}(r)-D^{\prime}(0)\right)}{-(N+2)D^{\prime\prime}(0)}\geq 0.$$
Together with the assumption \eqref{larn}, we obtain the desired result.
\end{proof}
We remark that the condition \eqref{larn} does not depend on the dimension $N$, which makes it easy to apply in the following. Indeed, if it holds for all $r>0$, then the Gaussian field and its derivatives associated to $D$ are nondegenerate, provided $D$ belongs to a certain class $\dx_N$ and Assumptions \ref{assumption1} and \ref{assumption2} hold.

\subsection{Representation with GOI matrices}

Under the nondegeneracy condition \eqref{nondg},
 we now investigate the expected number of critical points (with or without given indices) of non-isotropic Gaussian random fields with isotropic increments.

 We follow the idea for the case $\dx_\8$ in \cite{AZ20,AZ22} and list  the notations from \cite[(3.15)]{AZ20} for later use:
\begin{equation}
\label{eq:msialbt}
  \begin{split}
  m_1 & =m_1(\rho,u)=  \frac{u( 2D_N''(\rho^2)\rho^2+D_N'(\rho^2) -D_N'(0) )}{D_N(\rho^2) -\frac{D_N'(\rho^2)^2 \rho^2}{D_N'(0) }}, \\
  m_2&=m_2(\rho,u)=  \frac{u(D_N'(\rho^2) -D_N'(0) )}{D_N( \rho^2) -\frac{D_N'(\rho^2)^2 \rho^2}{D_N'(0) }}, \qquad \qquad
  \si_Y =\si_Y(\rho) =\sqrt{D_N(\rho^2)-\frac{D_N'(\rho^2)^2\rho^2}{D_N'(0)} }, \\
\alpha &=\alpha(\rho^2)= \frac{2D_N''(\rho^2)}{ \sqrt{ D_N(\rho^2)-\frac{D_N'(\rho^2)^2 \rho^2}{D_N'(0)}}},  \qquad \qquad \; \; \;
   \beta =\beta(\rho^2)=\frac{D_N'(\rho^2 )-D_N'(0)}{\sqrt{ D_N(\rho^2 )-\frac{D_N'(\rho^2)^2 \rho^2}{D_N'(0)}}}, \\
   \si_1 & =\si_1(\rho^2)= \sqrt{-4D_N''(0)-(\al \rho^2 +\beta)\al\rho^2}, \qquad
  \si_2 =\si_2(\rho^2)= \sqrt{-2D_N''(0)-(\al\rho^2 +\beta)\beta}.
  \end{split}
\end{equation}
Using Lemma \ref{Corv}, we define
\begin{align*}
    \Sigma_{01} &=\operatorname{Cov}\left(H_N(x), \nabla H_N(x)\right)=D_N^{\prime}\left(\|x\|^2\right) x^{\top}, \\
    \Sigma_{11} &=\operatorname{Cov}\left(\nabla H_N(x)\right)=D_N^{\prime}(0) \mathbf{I}_N,
\end{align*}
and
\begin{equation*}
    Y=H_N(x)-\Sigma_{01} \Sigma_{11}^{-1} \nabla H_N(x)=H_N(x)-\frac{D_N^{\prime}\left(\|x\|^2\right) \sum_{i=1}^N x_i \partial_i H_N(x)}{D_N^{\prime}(0)}.
\end{equation*}
Then $Y$ is a centered Gaussian random variable whose variance $\sigma_Y^2$ is defined in \eqref{eq:msialbt} with $\rho$ replaced by $\|x\|$.
Notice that both $Y$ and $\nabla^2 H_N(x)$ are independent of $\nabla H_N(x)$. By the Kac--Rice formula \cite[Theorem 11.2.1]{AT07},
\begin{equation}\label{Critical1}
    \begin{split}
    & \mathrel{\phantom{=}} \mathbb{E} [\operatorname{Crt}_N\left(E, T_N\right)] \\
    & =\int_{T_N} \mathbb{E}\left[\left|\operatorname{det} \nabla^2 H_N(x)\right| \mathbf{1}\left\{Y+ \Sigma_{01} \Sigma_{11}^{-1} \nabla H_N(x) \in E\right\} \middle\vert \nabla H_N(x)=0\right] p_{\nabla H_N(x)}(0) \, \mathrm{d} x \\
    & =\int_{T_N} \mathbb{E}\left[\left|\operatorname{det} \nabla^2 H_N(x)\right| \mathbf{1}\{Y \in E\}\right] p_{\nabla H_N(x)}(0) \, \mathrm{d} x,
    \end{split}
\end{equation}
and similarly
\begin{equation}\label{Critical2}
    \mathrel{\phantom{=}} \mathbb{E} [\operatorname{Crt}_{N,k} \left(E, T_N\right)]
   =\int_{T_N} \mathbb{E}\left[\left|\operatorname{det} \nabla^2 H_N(x)\right| \mathbf{1}\{Y \in E, i\left(\nabla^2 H_N(x)\right)=k \}\right] p_{\nabla H_N(x)}(0) \, \mathrm{d} x,
\end{equation}
where $p_{\nabla H_N(x)}(0) = (2 \pi D_N'(0))^{-N / 2}$ is the p.d.f.~of $\nabla H_N(x)$ at $0$.

From Lemma \ref{Corv}, we note that $\nabla^2 H_N(x)$ has the same distribution as $2\sqrt{-D_N^{\prime \prime}(0)}$ $\operatorname{GOI}(\frac{1}{2})$. Arguing as in \cite[Section 3]{AZ20} and using the spherical coordinates, we deduce that
\begin{align*}
    &\left(\nabla^2 H_N (x) \mid Y=u\right) \stackrel{d}{=}\begin{pmatrix}
      m_1 (\|x\|, u) & 0\\
0&m_2 (\|x\|, u) \mathbf{I}_{N-1}
    \end{pmatrix}
    +2 \sqrt{-D_N^{\prime \prime}(0)} M =: G,
\end{align*}
where $M$ is an  $\operatorname{SGOI} (d_1,d_2,d_3)$ matrix  with parameters
\begin{equation}\label{d123}
    d_1 = \frac{1}{2}+\frac{\beta (\|x\|^2) ^2}{4D_N'' (0)}, \qquad
    d_2 = \frac{\alpha (\|x\|^2) \beta (\|x\|^2) \|x\|^2}{4D_N'' (0)}, \qquad
    d_3 = \frac{\alpha (\|x\|^2)^2 \|x\|^4}{4D_N'' (0)},
\end{equation}
and the functions $m_1$, $m_2$, $\alpha$, and $\beta$ are given as in \eqref{eq:msialbt} with  $\rho$ replaced by $\|x\|$. Similar to \eqref{M}, we may write the shifted SGOI matrix $G$ as a block matrix
\begin{equation} \label{G}
    G = \begin{pmatrix}
    2\sqrt{- D_N^{\prime \prime}(0)}\zeta_1+m_1 & 2\sqrt{- D_N^{\prime \prime}(0)} \xi^{\top} \\
    2\sqrt{- D_N^{\prime \prime}(0)}\xi & G_{* *}
    \end{pmatrix},
\end{equation}
where
\begin{equation} \label{G**}
    G_{**} := 2\sqrt{- D_N^{\prime \prime}(0)} \left[ \mathrm{GOE}_{N-1} + \left( \zeta_2+\frac{m_2}{2\sqrt{-D_N^{\prime \prime}(0)}}\right) \mathbf{I}_{N-1} \right].
\end{equation}
From \eqref{goicond}, we know
\begin{equation} \label{y}
    \begin{split}
    \mathbb{E}(\GOE_{N-1}+\zeta_2\mathbf{I}_{N-1}\mid \zeta_1=y) &= \frac{(2D_N^{\prime \prime}(0)+\beta^2+\alpha\beta\|x\|^2)y}{6D_N^{\prime \prime}(0)+(\beta+\alpha\|x\|^2)^2}\mathbf{I}_{N-1}, \\
   \left( \GOE_{N-1}+\zeta_2\mathbf{I}_{N-1}\mid \zeta_1=y\right) &\stackrel{d}{=} \frac{(2D_N^{\prime \prime}(0)+\beta^2+\alpha\beta\|x\|^2)y}{6D_N^{\prime \prime}(0)+(\beta+\alpha\|x\|^2)^2}\mathbf{I}_{N-1}+M',
   \end{split}
\end{equation}
where $M'$ is an $(N-1)\times (N-1)$ GOI$(c)$ matrix with parameter
\begin{equation}\label{c}
    c = \frac{d_1+d_1d_3-d_2^2}{1+d_1+2d_2+d_3}=\frac{4D_N'' (0)+2\beta^2+\alpha^2\|x\|^4}{12D_N'' (0) + 2(\beta+\alpha\|x\|^2)^2}.
\end{equation}
 Set
\begin{equation}\label{m3}
m_3=\frac{(2D_N^{\prime \prime}(0)+\beta^2+\alpha\beta\|x\|^2)y}{6D_N^{\prime \prime}(0)+(\beta+\alpha\|x\|^2)^2}+\frac{m_2}{2\sqrt{-D_N^{\prime \prime}(0)}},
\end{equation}
where $m_2$ is defined in \eqref{eq:msialbt}. Let $\lambda_1 \leq \cdots \leq \lambda_{N-1}$ be the eigenvalues of the $\operatorname{GOI}(c)$ matrix $M'$. Since the distribution of $\operatorname{GOI}(c)$ matrices is invariant under orthogonal congruence transformation, following the same argument as for the GOE matrices, there exists a random orthogonal matrix $V$ independent of the unordered eigenvalues $\tilde{\lambda}_j, j=1, \ldots, N-1$ such that
\begin{align}\label{rotation}
V M' V^\mathsf T = \begin{pmatrix}
 \tilde{\lambda}_1 & \cdots & 0 \\
\vdots & \ddots & \vdots \\
0 & \cdots & \tilde{\lambda}_{N-1}
\end{pmatrix}.
\end{align}
Since the law of the Gaussian vector $\xi$ is rotationally invariant, $V \xi$ is a centered Gaussian vector with covariance matrix $\frac{1}{2} \mathbf{I}_{N-1}$ that is independent of $\tilde{\lambda}_j$'s. We can rewrite $V \xi \stackrel{d}{=} Z / \sqrt{2}$, where $Z = (Z_1, \ldots, Z_{N-1} )$ is an $(N-1)$-dimensional standard Gaussian random vector. We also need the following lemma for the calculation of  $\mathbb{E} [\operatorname{Crt}_{N,k}\left(E, \left(R_{1}, R_{2}\right)\right)]$. Recall that the signature of a symmetric matrix is the number of positive eigenvalues minus that of negative eigenvalues.
\begin{lemma}[{\cite[Equation 2]{Laz88}}]
Let $S$ be a symmetric block matrix, and write its inverse $S^{-1}$ in block form with the same block structure:
$$
S=\left(\begin{array}{cc}
A & B \\
B^{\sfT} & C
\end{array}\right), \qquad S^{-1}=\left(\begin{array}{cc}
A^{\prime} & B^{\prime} \\
\left(B^{\prime}\right)^{\sfT} & C^{\prime}
\end{array}\right) .
$$
Then $\operatorname{sgn}(S)=\operatorname{sgn}(A)+\operatorname{sgn}\left(C^{\prime}\right)$, with $\operatorname{sgn}(M)$ denoting the signature of the matrix $M$.\label{sgn}
\end{lemma}
 Let us define
\begin{equation*}
    \eta (\zeta_1,G_{**}) =m_1+2\sqrt{-D_N^{\prime \prime}(0)}\zeta_1+4D_N^{\prime \prime}(0)\xi^{\top} G_{* *}^{-1} \xi.
\end{equation*}
Following Lemma \ref{sgn}, we have for $1\leq k \leq N-1$,
\begin{equation*}
    \{i(G)=k\}=\left\{i\left(G_{* *}\right)=k, \eta (\zeta_1,G_{**}) >0\right\} \cup \left\{ i\left(G_{* *}\right)=k-1, \eta (\zeta_1,G_{**}) <0 \right\},
\end{equation*}
for $k=0$,
\begin{equation*}
    \{i(G)=0\}=\left\{i\left(G_{* *}\right)=0, \eta (\zeta_1,G_{**}) >0\right\},
\end{equation*}
and for $k=N$,
\begin{equation*}
    \{i(G)=N\}=\left\{ i\left(G_{* *}\right)=N-1, \eta (\zeta_1,G_{**}) <0 \right\} .
\end{equation*}
Moreover, by \eqref{y} and \eqref{rotation}, we deduce
\begin{equation} \label{eta'}
    \begin{split}
    \left(\eta(\zeta_1,G_{**}) \middle\vert \zeta_1=y\right)
    &\stackrel{d}{=} m_1+2\sqrt{-D_N^{\prime \prime}(0)}y-2\sqrt{-D_N^{\prime \prime}(0)}\xi^{\top} \left(M'+m_3\mathbf{I}_{N-1}\right)^{-1} \xi \\
    &\stackrel{d}{=}m_1+2\sqrt{-D_N^{\prime \prime}(0)}y-\sqrt{-D_N^{\prime \prime}(0)}\sum_{l=1}^{N-1} \frac{Z_l^2}{\lambda_l+m_3} \\
    &=: \eta'(\Lambda),
    \end{split}
\end{equation}
where $\Lambda=(\lambda_1,\dots,\lambda_{N-1}).$ 

Recall that $M$ is an  $\operatorname{SGOI} (d_1,d_2,d_3)$ matrix with parameters $d_1,d_2$ and $d_3$ defined in \eqref{d123} depending on the structure function $D_N$. Since we will use $M$ to deduce the  main theorem in this section, it is also necessary to guarantee that $M$ is nondegenerate. 
\begin{lemma}
The nondegeneracy condition \eqref{nondg} holds for $r=\|x\|^{2}$ if and only if the $\operatorname{SGOI} (d_1,d_2,d_3)$ matrix $M$ is nondegenerate for $x \in \mathbb{R}^N \setminus \{0\}$, where $d_1$, $d_2$ and $d_3$ are defined in \eqref{d123}. 
\end{lemma}
\begin{proof}
    According to Lemma \ref{NonPGOI}, it suffices to show that \eqref{nondg} is equivalent to
    \begin{equation*}
        d_1 > -\frac{1}{N-1}, \qquad \qquad
        1 + d_3 + \frac{d_1 + 2d_2 - (N-1) d_2^2}{1 + (N-1) d_1} > 0,
    \end{equation*}
    where
   $d_1 = \frac{1}{2}+\frac{\beta (r) ^2}{4D_N'' (0)},  d_2 = \frac{\alpha (r) \beta (r) r}{4D_N'' (0)}, 
    d_3 = \frac{\alpha (r)^2 r^2}{4D_N'' (0)}$
   with $\alpha(r)$ and $\beta(r)$ defined in \eqref{eq:msialbt}, and we have replaced $\lVert x \rVert^2$ with $r$ in the definitions of $d_1$, $d_2$ and $d_3$ for simplicity.
    
  If $d_1 > -\frac{1}{N-1}$, straightforward calculations show that $1 + d_3 + \frac{d_1 + 2d_2 - (N-1) d_2^2}{1 + (N-1) d_1} > 0$ is equivalent to \eqref{nondg}. To be precise, multiplying both sides by $1 + (N-1) d_1$ implies
\begin{equation*}
    1 + d_3 + (N-1) (d_1 d_3 - d_2^2) + 2 d_2 + N d_1 > 0.
\end{equation*}
Note that $d_1 d_3 - d_2^2 = \frac{1}{2} d_3$ and
\begin{equation*}
    1 + \frac{N+1}{2} d_3 + 2 d_2 + N d_1
    = 1 + \frac{N}{2} + \frac{(N+1) \alpha (r)^2 r^2}{8 D_N'' (0)} + \frac{\alpha (r) \beta (r) r}{2 D_N'' (0)} + \frac{N \beta (r)^2}{4 D_N'' (0)} > 0.
\end{equation*}
Multiplying both sides by $\frac{2}{N+2} ( D_N(r)-\frac{D_N'(r)^2 r}{D_N'(0)} )$,
we get
\begin{equation*}
    D_N(r)-\frac{D_N'(r)^2 r}{D_N'(0)} + \frac{(N+1) D_N'' (r)^2 r^2}{(N+2) D_N'' (0)} + \frac{2 r D_N'' (r) (D_N' (r) - D_N' (0))}{(N+2) D_N'' (0)} + \frac{(D_N' (r) - D_N' (0))^2}{2 (N+2) D_N'' (0)} > 0,
\end{equation*}
which is exactly \eqref{nondg}. Notice that each step of above derivations is reversible. Therefore, \eqref{nondg} is equivalent to $1 + d_3 + \frac{d_1 + 2d_2 - (N-1) d_2^2}{1 + (N-1) d_1} > 0$, on condition that $d_1 > -\frac{1}{N-1}$.

It remains to show that $d_1 > -\frac{1}{N-1}$ follows from \eqref{nondg}. Dividing both sides of \eqref{nondg} by $\frac{2}{N+2} ( D_N(r)-\frac{D_N'(r)^2 r}{D_N'(0)} )$ gives that
\begin{equation*}
    1 + \frac{N}{2} + \frac{(N+1) \alpha (r)^2 r^2}{8 D_N'' (0)} + \frac{\alpha (r) \beta (r) r}{2 D_N'' (0)} + \frac{N \beta (r)^2}{4 D_N'' (0)} > 0,
\end{equation*}
which is equivalent to
\begin{equation*}
    1 + \frac{N}{2} + \frac{(N+2) (N-1) \beta (r)^2}{4 (N+1) D_N'' (0)} + \frac{[(N+1) \alpha (r) r + 2 \beta (r)]^2}{8 (N+1) D_N'' (0)} > 0.
\end{equation*}
Since $D_N'' (0) < 0$ implies $\frac{[(N+1) \alpha (r) r + 2 \beta (r)]^2}{8 (N+1) D_N'' (0)} < 0$, we obtain
\begin{equation*}
    1 + \frac{N}{2} + \frac{(N+2) (N-1) \beta (r)^2}{4 (N+1) D_N'' (0)} > 0.
\end{equation*}
Multiplying both sides by $\frac{N+1}{(N+2)(N-1)}$  implies
\begin{equation*}
    \frac{N+1}{2(N-1)} + \frac{\beta (r)^2}{4 D_N'' (0)} = \frac{1}{N-1} + \frac{1}{2} + \frac{\beta (r)^2}{4 D_N'' (0)} > 0,
\end{equation*}
which shows $d_1 > -\frac{1}{N-1}$ and thus completes the proof.  
\end{proof}

Due to the rotational symmetry of isotropy, we are interested in the critical points in the shell domain
$$T_{N}\left(R_{1}, R_{2}\right)=\{x \in \mathbb{R}^{N}: R_{1}<\|x\|<R_{2}\},\quad  0 \leq R_{1}<R_{2} <\infty,$$
with critical values in an arbitrary Borel subset $E$. For simplicity, we write $\operatorname{Crt}_{N}\left(E,\left(R_{1}, R_{2}\right)\right):=$ $\operatorname{Crt}_{N}\left(E, T_{N}\left(R_{1}, R_{2}\right)\right)$. Together with \eqref{Critical1} and \eqref{Critical2}, using the spherical coordinates and writing $\rho=\|x\|$, we arrive at
\begin{align}
\mathbb{E} [\operatorname{Crt}_N\left(E, \left(R_{1}, R_{2}\right)\right)]
&=S_{N-1}\int_{R_1}^{R_2} \int_{E}\mathbb{E}\left[\left|\operatorname{det} G\right| \right] \frac{1}{\sqrt{2 \pi} \sigma_{Y}} e^{-\frac{u^{2}}{2 \sigma_{Y}^{2}}}p_{\nabla H_N(x)}(0) \rho^{N-1} \, \dd u \, \mathrm{d} \rho, \label{R1R2_1} \\
\mathbb{E} [\operatorname{Crt}_{N,k} \left(E, \left(R_{1}, R_{2}\right)\right)]
&=S_{N-1}\int_{R_1}^{R_2} \int_{E}\mathbb{E}\left[\left|\operatorname{det} G\right|  \mathbf{1}\{ i(G) = k \} \right] \frac{1}{\sqrt{2 \pi} \sigma_{Y}} e^{-\frac{u^{2}}{2 \sigma_{Y}^{2}}} \notag \\
&\phantom{XXXXXXXXXXXXXXXXXXXXXx} \times p_{\nabla H_N(x)}(0) \rho^{N-1} \, \dd u \, \mathrm{d} \rho, \label{R1R2_2}
\end{align}
where $S_{N-1} = \frac{2 \pi^{N / 2}} { \Gamma(N / 2)}$ is the area of $(N-1)$-dimensional unit sphere; see \cite[(4.1)]{AZ20} for details.

Unlike the isotropic case in \cite{CS18}, here it is difficult to find the joint eigenvalue density of the matrix $G$ due to the lack of invariance property. However, we observe that conditioning on the first entry of an SGOI($d_1,d_2,d_3$) matrix, its lower right $(N-1)\times(N-1)$ submatrix has the same distribution as a GOI$(c)$ matrix for some suitable $c$. We are now ready to state our main result in this section.

\begin{theorem} \label{Shell}
Let $H_N= \{H_N(x): x \in \mathbb{R}^N \}$ be a non-isotropic Gaussian field with isotropic increments. Let $E \subset \mathbb{R}$ be a Borel set and $T_N$ be the shell domain $T_{N}\left(R_{1}, R_{2}\right)=\{x \in \mathbb{R}^{N}: R_{1}<\|x\|<R_{2}\}$, where $0 \leq R_{1}<R_{2} < \infty$.
Assume Assumptions \ref{assumption1}, \ref{assumption2} and the nondegeneracy condition \eqref{nondg} holds for all $r>0$. Then we have
\begin{align*}
    &\mathrel{\phantom{=}} \mathbb{E} [\operatorname{Crt}_N\left(E, \left(R_{1}, R_{2}\right)\right)] \\
    & =\frac{2\left(-2D_N''(0)\right)^{N/2}}{ D_N^{\prime}(0)^{N / 2}\Gamma(N / 2)}\int_{R_1}^{R_2} \int_{E}\int_{-\infty}^{\infty}\frac{e^{-\frac{2D_N''(0)y^{2}}{ 6D_N''(0)+(\beta+\alpha\rho^2)^2}}}{\sqrt{2 \pi [-6D_N''(0)-(\beta+\alpha\rho^2)^2]}}  \frac{ e^{-\frac{u^{2}}{2 \sigma_{Y}^{2}}}}{\sqrt{2 \pi} \sigma_{Y}} \rho^{N-1} \, \dd y \, \dd u \, \mathrm{d} \rho \\
    & \quad\times \mathbb{E}_Z\left\{ \mathbb{E}_{\mathrm{GOI}(c)} \left[ \left|\left(m_1+2\sqrt{-D_N^{\prime \prime}(0)}y\right)\prod_{j=1}^{N-1}(\lambda_j+m_3)-\sqrt{-D_N^{\prime \prime}(0)}\sum_{l=1}^{N-1} Z_l^2 \prod_{j \neq l}^{N-1}(\lambda_j+m_3)\right| \right] \right\},
\end{align*}
and for $k=0,1,\dots,N$,
\begin{align*}
    &\mathrel{\phantom{=}} \mathbb{E} [\operatorname{Crt}_{N,k} \left(E, \left(R_{1}, R_{2}\right)\right)] \\
    & =\frac{2\left(-2D_N''(0)\right)^{N/2}}{ D_N^{\prime}(0)^{N / 2}\Gamma(N / 2)}\int_{R_1}^{R_2} \int_{E}\int_{-\infty}^{\infty}\frac{e^{-\frac{2D_N''(0)y^{2}}{ 6D_N''(0)+(\beta+\alpha\rho^2)^2}}}{\sqrt{2 \pi [-6D_N''(0)-(\beta+\alpha\rho^2)^2]}}  \frac{ e^{-\frac{u^{2}}{2 \sigma_{Y}^{2}}}}{\sqrt{2 \pi} \sigma_{Y}} \rho^{N-1} \, \dd y \, \dd u \, \mathrm{d} \rho \\
    &\; \; \times \mathbb{E}_{Z}\left\{ \mathbb{E}_{\mathrm{GOI}(c)} \left[ \left|\left(m_1+2\sqrt{-D_N^{\prime \prime}(0)}y\right)\prod_{j=1}^{N-1}(\lambda_j+m_3)-\sqrt{-D_N^{\prime \prime}(0)}\sum_{l=1}^{N-1} Z_l^2 \prod_{j \neq l}^{N-1}(\lambda_j+m_3)\right| \mathbf{1}_{A_k} \right] \right\},
\end{align*}
where
\begin{align*}
    A_k = \begin{cases}
        \left\{\lambda_k<-m_3<\lambda_{k+1}, \eta' (\Lambda)>0 \right\} \cup \left\{\lambda_{k-1}<-m_3<\lambda_{k}, \eta'(\Lambda)<0\right\}, &1\leq k \leq N-1, \\
        \left\{\lambda_1 > -m_3, \eta' (\Lambda)>0 \right\}, & k = 0,\\
        \left\{\lambda_{N-1} < -m_3, \eta' (\Lambda)<0 \right\}, & k = N,
    \end{cases}
\end{align*}
$\eta' (\Lambda)$ is defined in \eqref{eta'}, $m_1$ and $\sigma_Y$ are given in \eqref{eq:msialbt}, $c$ is defined in \eqref{c}, $m_3$ is defined in \eqref{m3}, $Z = (Z_1, \ldots, Z_{N-1} )$ is an $(N-1)$-dimensional standard Gaussian random vector independent of $\la_j$'s, and by convention $\lambda_0=-\infty$, $\lambda_N=\infty$. We emphasize that the expectation $\mathbb{E}_{\mathrm{GOI}(c)}$ is taken with respect to the GOI$(c)$ eigenvalues $\lambda_i$'s.
\end{theorem}
\begin{proof}
Recall that $G = \left( \begin{smallmatrix}
      m_1 &0\\
0&m_2\mathbf{I}_{N-1}
\end{smallmatrix} \right) + 2\sqrt{-D_N^{\prime \prime}(0)}M$, where $M$ is an $\operatorname{SGOI} (d_1,d_2,d_3)$ matrix defined in \eqref{M} with parameters $d_1=\frac{1}{2}+\frac{\beta^2}{4D_N^{\prime \prime}(0)}$, $d_2=\frac{\alpha\beta\rho^2}{4D_N^{\prime \prime}(0)}$ and $d_3=\frac{\alpha^2\rho^4}{4D_N^{\prime \prime}(0)}$.
The Schur complement formula implies that
\begin{align}\label{detGG}
    \operatorname{det} G&=\left(m_1+2\sqrt{-D_N^{\prime \prime}(0)}\zeta_1+4D_N'' (0)\xi^{\top} G_{* *}^{-1} \xi\right)\det\left(G_{* *}\right) \nonumber \\
    &= \left\{ m_1+2\sqrt{-D_N^{\prime \prime}(0)}\zeta_1-2\sqrt{-D_N'' (0)}\xi^\mathsf T \left[ \left(\frac{m_2}{2\sqrt{-D_N'' (0)}}+\zeta_2\right)\mathbf{I}_{N-1}+\GOE_{N-1} \right]^{-1} \xi \right\} \nonumber \\
    &\quad \times2^{N-1}\left(-D_N'' (0)\right)^\frac{N-1}{2}\operatorname{det} \left(\left(\frac{m_2}{2\sqrt{-D_N'' (0)}}+\zeta_2\right)\mathbf{I}_{N-1}+\GOE_{N-1}\right).
\end{align}
Conditioning on $\zeta_1=y$, 
using \eqref{y}, \eqref{m3} and \eqref{detGG}, we obtain
\begin{equation} \label{EdetG}
    \begin{split}
&\mathbb{E}\left[\left|\operatorname{det} G\right| \right]
    = \int_{-\infty}^{\infty}\frac{2^{N-1} \left( -D_N'' (0) \right)^\frac{N-1}{2}}{\sqrt{2 \pi(1+d_1+2d_2+d_3)}} e^{-\frac{y^{2}}{2 (1+d_1+2d_2+d_3)}} \\
    &\phantom{\operatorname{det} G} \times \mathbb{E} \left[ \left|\left(m_1+2\sqrt{-D_N^{\prime \prime}(0)}y-2\sqrt{-D_N^{\prime \prime}(0)}\xi^\mathsf T \left(M'+m_3\mathbf{I}_{N-1}\right)^{-1}\xi\right)\operatorname{det} \left(M'+m_3\mathbf{I}_{N-1}\right)\right| \right] \dd y,
    \end{split}
\end{equation}
where $1+d_1+2d_2+d_3=\frac{6D_N'' (0) + (\beta+\alpha\rho^2)^2}{4D_N'' (0)}$ and $M'$ is the $(N-1)\times (N-1)$ GOI$(c)$ matrix with parameter $c$ defined in \eqref{c}.
By \eqref{rotation} and the independence of $Z$ and $\lambda_i$'s, we deduce
\begin{equation} \label{EdetGG}
    \begin{split}
    &\mathrel{\phantom{=}} \mathbb{E}\left[\left|\left(m_1+2\sqrt{-D_N^{\prime \prime}(0)}y-2\sqrt{-D_N^{\prime \prime}(0)}\xi^\mathsf T \left( M' + m_3\mathbf{I}_{N-1}\right)^{-1}\xi\right)\operatorname{det} \left( M' + m_3\mathbf{I}_{N-1}\right)\right| \right] \\
    &=\mathbb{E}\left[\left|\left(m_1+2\sqrt{-D_N^{\prime \prime}(0)}y\right)\prod_{j=1}^{N-1}(\lambda_j+m_3)-\sqrt{-D_N^{\prime \prime}(0)}\sum_{l=1}^{N-1} Z_l^2 \prod_{j \neq l}^{N-1}(\lambda_j+m_3)\right| \right] \\
    &=\mathbb{E}_{Z} \left\{ \mathbb{E}_{\mathrm{GOI}(c)} \left[ \left| \left(m_1+2\sqrt{-D_N^{\prime \prime}(0)}y\right)\prod_{j=1}^{N-1}(\lambda_j+m_3)-\sqrt{-D_N^{\prime \prime}(0)}\sum_{l=1}^{N-1} Z_l^2 \prod_{j \neq l}^{N-1}(\lambda_j+m_3)\right| \right] \right\}.
    \end{split}
\end{equation}
Combining  \eqref{R1R2_1}, \eqref{EdetG} and \eqref{EdetGG}, after some simplifications we obtain
\begin{align*}
&\mathrel{\phantom{=}} \mathbb{E} [\operatorname{Crt}_N\left(E, \left(R_{1}, R_{2}\right)\right)] \\
& =\frac{2\left(-2D_N''(0)\right)^{N/2}}{ D_N^{\prime}(0)^{N / 2}\Gamma(N / 2)}\int_{R_1}^{R_2} \int_{E}\int_{-\infty}^{\infty}\frac{e^{-\frac{2D_N''(0)y^{2}}{ 6D_N''(0)+(\beta+\alpha\rho^2)^2}}}{\sqrt{2 \pi [-6D_N''(0)-(\beta+\alpha\rho^2)^2]}}  \frac{ e^{-\frac{u^{2}}{2 \sigma_{Y}^{2}}}}{\sqrt{2 \pi} \sigma_{Y}} \rho^{N-1} \, \dd y \, \dd u \, \dd \rho \\
& \quad \times \mathbb{E}_{Z}\left\{ \mathbb{E}_{\mathrm{GOI}(c)} \left[ \left|\left(m_1+2\sqrt{-D_N^{\prime \prime}(0)}y\right)\prod_{j=1}^{N-1}(\lambda_j+m_3)-\sqrt{-D_N^{\prime \prime}(0)}\sum_{l=1}^{N-1} Z_l^2 \prod_{j \neq l}^{N-1}(\lambda_j+m_3)\right| \right] \right\},
\end{align*}
which is the desired result for the first part.

Let us turn to the second assertion.
For $1\leq k \leq N-1$, combining \eqref{eta'}, \eqref{detGG}, \eqref{EdetG} and \eqref{EdetGG} gives
\begin{align*}
    &\mathrel{\phantom{=}} \mathbb{E} \left[ \left|\operatorname{det} G\right|  \mathbf{1}\{ i(G) = k \} \right] \\
    &= \mathbb{E} \left[ \left|\operatorname{det} G\right|  \left( \mathbf{1}\{ i\left(G_{* *}\right)=k, \eta (\zeta_1,G_{**}) >0 \} + \mathbf{1}\{ i\left(G_{* *}\right)=k-1, \eta (\zeta_1,G_{**}) <0 \} \right) \right] \\
    &= \int_{-\infty}^{\infty} \frac{2^{N-1} \left( -D_N'' (0) \right)^\frac{N-1}{2}}{\sqrt{2 \pi(1+d_1+2d_2+d_3)}} e^{-\frac{y^{2}}{2 (1+d_1+2d_2+d_3)}} \\
    &\quad \times \mathbb{E}_Z \left\{ \mathbb{E}_{\mathrm{GOI}(c)} \left[ \left| \left(m_1+2\sqrt{-D_N'' (0)}y\right)\prod_{j=1}^{N-1}(\lambda_j+m_3)-\sqrt{-D_N'' (0)}\sum_{l=1}^{N-1} Z_l^2 \prod_{j \neq l}^{N-1}(\lambda_j+m_3)\right| \right. \right. \\
    &\phantom{XXXXXXX} \left. \left. \vphantom{\left| \prod_{j=1}^{N-1} \right|} \times \left( \mathbf{1} \left\{\lambda_k<-m_3<\lambda_{k+1}, \eta' (\Lambda)>0 \right\} + \mathbf{1} \left\{\lambda_{k-1}<-m_3<\lambda_{k}, \eta'(\Lambda)<0\right\} \right) \right] \right\} \dd y.
\end{align*}
Plugging the above equation into \eqref{R1R2_2} yields that
\begin{align*}
    &\mathrel{\phantom{=}} \mathbb{E} [\operatorname{Crt}_{N,k}\left(E, \left(R_{1}, R_{2}\right)\right)] \\
    &= \frac{2\left(-2D_N''(0)\right)^{N/2}}{ D_N^{\prime}(0)^{N / 2}\Gamma(N / 2)}\int_{R_1}^{R_2} \int_{E}\int_{-\infty}^{\infty}\frac{e^{-\frac{2D_N''(0)y^{2}}{ 6D_N''(0)+(\beta+\alpha\rho^2)^2}}}{\sqrt{2 \pi [-6D_N''(0)-(\beta+\alpha\rho^2)^2]}}  \frac{ e^{-\frac{u^{2}}{2 \sigma_{Y}^{2}}}}{\sqrt{2 \pi} \sigma_{Y}} \rho^{N-1} \, \dd y \, \dd  u \, \dd \rho \\
    & \quad\times \mathbb{E}_{Z}\left\{\mathbb{E}_{\mathrm{GOI}(c)}\left[\left|\left(m_1+2\sqrt{-D_N^{\prime \prime}(0)}y\right)\prod_{j=1}^{N-1}(\lambda_j+m_3)-\sqrt{-D_N^{\prime \prime}(0)}\sum_{l=1}^{N-1} Z_l^2 \prod_{j \neq l}^{N-1}(\lambda_j+m_3)\right| \mathbf{1}_{A_k} \right] \right\},
\end{align*}
where $A_k = \{ \lambda_k<-m_3<\lambda_{k+1}, \eta' (\Lambda)>0 \} \cup \{ \lambda_{k-1}<-m_3<\lambda_{k}, \eta'(\Lambda)<0 \}$ for $1\leq k \leq N-1$.

Similarly, for $k=0$ we have
\begin{align*}
    &\mathrel{\phantom{=}} \mathbb{E} \left[ \left|\operatorname{det} G\right|  \mathbf{1}\{ i(G) = 0 \} \right] \\
    &= \mathbb{E} \left[ \left|\operatorname{det} G\right|  \mathbf{1}\{ i\left(G_{* *}\right)=0, \eta (\zeta_1,G_{**}) >0 \} \right] \\
    &= \int_{-\infty}^{\infty} \frac{2^{N-1} \left( -D_N'' (0) \right)^\frac{N-1}{2}}{\sqrt{2 \pi(1+d_1+2d_2+d_3)}} e^{-\frac{y^{2}}{2 (1+d_1+2d_2+d_3)}} \mathbb{E}_Z \left\{ \mathbb{E}_{\mathrm{GOI}(c)} \left[ \vphantom{\left| \prod_{j=1}^{N-1} \right|} \mathbf{1} \left\{\lambda_1 > -m_3, \eta' (\Lambda)>0 \right\} \right. \right. \\
    &\quad \times \left. \left. \left| \left(m_1+2\sqrt{-D_N'' (0)}y\right)\prod_{j=1}^{N-1}(\lambda_j+m_3)-\sqrt{-D_N'' (0)}\sum_{l=1}^{N-1} Z_l^2 \prod_{j \neq l}^{N-1}(\lambda_j+m_3)\right| \right] \right\} \dd y,
\end{align*}
and
\begin{align*}
    &\mathrel{\phantom{=}} \mathbb{E} [\operatorname{Crt}_{N,0}\left(E, \left(R_{1}, R_{2}\right)\right)] \\
    &= \frac{2\left(-2D_N''(0)\right)^{N/2}}{ D_N^{\prime}(0)^{N / 2}\Gamma(N / 2)}\int_{R_1}^{R_2} \int_{E}\int_{-\infty}^{\infty}\frac{e^{-\frac{2D_N''(0)y^{2}}{ 6D_N''(0)+(\beta+\alpha\rho^2)^2}}}{\sqrt{2 \pi [-6D_N''(0)-(\beta+\alpha\rho^2)^2]}}  \frac{ e^{-\frac{u^{2}}{2 \sigma_{Y}^{2}}}}{\sqrt{2 \pi} \sigma_{Y}} \rho^{N-1} \, \dd y \, \dd  u \, \dd \rho \\
    & \quad\times \mathbb{E}_Z \left\{ \mathbb{E}_{\mathrm{GOI}(c)} \left[ \left|\left(m_1+2\sqrt{-D_N^{\prime \prime}(0)}y\right)\prod_{j=1}^{N-1}(\lambda_j+m_3)-\sqrt{-D_N^{\prime \prime}(0)}\sum_{l=1}^{N-1} Z_l^2 \prod_{j \neq l}^{N-1}(\lambda_j+m_3)\right| \mathbf{1}_{A_0} \right] \right\},
\end{align*}
where $A_0 = \{\lambda_1 > -m_3, \eta' (\Lambda)>0 \}$. We omit the proof for index $k=N$ here, since it is similar to the case $k=0$.
\end{proof}

\subsection{Representation with GOE matrices}
In order for the large $N$ asymptotic analysis, it is desirable to write the GOI matrix in the above representation as a sum of a GOE matrix and an independent scalar matrix. Most of results in this subsection follow from arguments similar to those in \cite{AZ20,AZ22}. First, as observed in these works, if the critical values are not restricted (i.e., $E=\rz$), there is no need to consider the conditional distribution of Hessian and it is straightforward to employ GOE matrices for the Kac--Rice representation.
\begin{theorem}\label{ER}
Let $H_N=\left\{H_N(x): x \in \mathbb{R}^N\right\}$ be a non-isotropic Gaussian field with isotropic increments and $T_N$ be a Borel subset of $\mathbb{R}^{N}$.
Assume Assumptions \ref{assumption1} and \ref{assumption2}, we have
\begin{align*}
 \mathbb{E} [\operatorname{Crt}_N\left( \mathbb{R}, T_N\right)] &= \frac{\left(-2D_N''(0)\right)^{N/2}|T_N|}{\pi^{(N+1)/2} D_N^{\prime}(0)^{N / 2}}\int_{-\infty}^{\infty}\mathbb{E}_{\mathrm{GOE}}\left[ \prod_{j=1}^N\left|\lambda_j+y\right|\right]e^{-y^2}\mathrm{d} y, \\
\mathbb{E} [\operatorname{Crt}_{N,k}\left(\mathbb{R},T_N\right)] &= \frac{\left(-2D_N''(0)\right)^{N/2}|T_N|}{\pi^{(N+1)/2} D_N^{\prime}(0)^{N / 2}}  \int_{-\infty}^{\infty} \mathbb{E}_{\mathrm{GOE}} \left[ \prod_{j=1}^N\left|\lambda_j+y\right|\mathbf{1}{ \{\lambda_k<-y<\lambda_{k+1}\} }\right] e^{-y^2}  \mathrm{d} y, 
\end{align*}
where $|T_N|$ is the Lebesgue measure of $T_N$, $0\leq k\leq N,$ $\lambda_1 \leq \cdots \leq \lambda_N$ are the ordered eigenvalues of the GOE matrix $\GOE_{N}$, and by convention $\lambda_0=-\infty$, $\lambda_{N+1}=\infty$.
\end{theorem}
The proof is omitted here since it is the same as for \cite[Theorem 1.1]{AZ20} and \cite[Theorems 1.1, 1.2]{AZ22}. As a comparison to the isotropic case, we give an example below, which is analogous to {\cite[Example 3.8]{CS18}}. 
\begin{example}
Let the assumptions in Theorem \ref{ER} hold. Consider $N=2$ and $T_2 \subset \mathbb{R}^2$ with unit area. Using  Theorem \ref{ER}, we obtain
$$\mathbb{E} [\operatorname{Crt}_{2,k}\left(\mathbb{R},T_2\right)] = \frac{-2D_2''(0)}{\pi^{3/2} D_2^{\prime}(0)}  \int_{-\infty}^{\infty} \mathbb{E}_{\mathrm{GOE}} \left[ \prod_{j=1}^2\left|\lambda_j+y\right|\mathbf{1}{ \{\lambda_k<-y<\lambda_{k+1}\} }\right] e^{-y^2}  \mathrm{d} y.$$  
Replacing the $\mathrm{GOE}$ density with the $\mathrm{GOI(c)}$ density with $c=\frac{1}{2}$, we deduce
$$\mathbb{E} [\operatorname{Crt}_{2,k}\left(\mathbb{R},T_2\right)] = \frac{-2D_2''(0)}{\pi D_2^{\prime}(0)}  \mathbb{E}_{\mathrm{GOI}(\frac{1}{2})} \left[ \prod_{j=1}^2\left|\lambda_j\right|\mathbf{1}{ \{\lambda_k<0<\lambda_{k+1}\} }\right].$$
Plugging the $\mathrm{GOI}(c)$ density (\ref{eq:goidens}) with $N=2$ and $c=\frac{1}{2}$ into the above equality implies
$$
\mathbb{E} [\operatorname{Crt}_{2,0}\left(\mathbb{R},T_2\right)]
= \mathbb{E} [\operatorname{Crt}_{2,2}\left(\mathbb{R},T_2\right)]
= \frac{1}{2} \mathbb{E}[\operatorname{Crt}_{2,1}\left(\mathbb{R},T_2\right)]
= \frac{-D_2''(0)}{\sqrt{3}\pi D_2'(0)}.$$  
\end{example}

The situation is more involved when the critical values are constrained.
The following condition for any fixed $N$ turns out to be sufficient for using GOE matrices in the representation.

\begin{assumption} \label{assumption3}
    For any $r>0$, we have 
\begin{align} \label{ar}
-2 D_N^{\prime \prime}(0) &>\left(\alpha r+\beta\right) \beta, \\
-4 D_N^{\prime \prime}(0) &>\left(\alpha r+\beta\right) \alpha r, \label{be}\\
\alpha\beta &>0, \nonumber
\end{align}
where $\alpha=\alpha\left(r \right)$ and $\beta=\beta\left(r \right)$ are defined in \eqref{eq:msialbt} with $\rho^2$ replaced by $r$.
\end{assumption}

This condition was used in \cite{AZ20} for the structure functions in the class $\dx_\8$, and in that case we trivially have $\alpha(r)\beta(r)>0$ for any $r>0$. Actually, we will show that Assumption \ref{assumption3} holds for all structure functions in $\dx_\8$ in the next section. 
A natural question is the relationship between Assumption \ref{assumption3} and the nondegeneracy of the field and its derivatives. The following result shows that Assumption \ref{assumption3} is stronger than the nondegeneracy condition \eqref{nondg}.
Recall the parameter $c=\frac{4D_N'' (0)+2\beta^2+\alpha^2r^2}{12D_N'' (0) + 2(\beta+\alpha r)^2}$ with $\|x\|^2$ replaced by $r$ from \eqref{c}. Assumption \ref{assumption3} also implies $c>0$.
\begin{lemma}
  For all fixed $N\in\nz$, if Assumption \ref{assumption3} holds, then we have $c>0$ and the Gaussian vector
\begin{equation*}
	\left(H_N(x), \nabla H_N(x), \partial_{i j} H_{N}(x),1\leq i\leq j\leq N\right),
\end{equation*}
is nondegenerate.
\end{lemma}
\begin{proof}
 First of all, \eqref{ar} and \eqref{be} imply that the denominator of $c$ is smaller than $0$. To prove $c>0$, it remains to show that 
 \begin{align}\label{eq:dalbt}
 -4D_N''(0) > 2\beta^2 + \alpha^2r^2,
 \end{align}
 which is equivalent to the inequality (\ref{larn}) and thus further yields the nondegeneracy condition by Lemma \ref{Nond}. Since $\alpha \beta > 0$, we have $$(2 \alpha \beta r - \alpha^2 r^2) (\alpha \beta r - 2 \beta^2) = -\alpha \beta r (2 \beta - \alpha r)^2 \leq 0.$$ 
 If $2 \alpha \beta r - \alpha^2 r^2 \geq 0$, then \eqref{ar} implies that
 \begin{equation*}
     -4 D_N'' (0) > 2 \alpha \beta r + 2 \beta^2 \geq \alpha^2 r^2 + 2 \beta^2.
 \end{equation*}
 If $2 \alpha \beta r - \alpha^2 r^2 < 0$, we must have $\alpha \beta r - 2 \beta^2 \geq 0$ and \eqref{be} implies that
 \begin{equation*}
     -4 D_N'' (0) > \alpha^2 r^2 + \alpha \beta r \geq \alpha^2 r^2 + 2 \beta^2. \qedhere
 \end{equation*}
\end{proof}

Using this lemma, we may deduce the GOE representation from Theorem \ref{Shell} by unraveling  GOE eigenvalues and an independent Gaussian random variable from the GOI$(c)$ matrix. This is equivalent to considering the conditional distribution of $z_3'$ given $z_1'$ below in \pref{eq:cond}, which would bring in an extra random variable in \eqref{aN}. For transparency and for a different perspective compared with Theorem \ref{Shell}, here we provide formulas that are suitable for asymptotic analysis following \cite{AZ20,AZ22}. Since Assumption \ref{assumption3} implies $d_1>0$, according to \cite[Section 3]{AZ20}, the shifted SGOI matrix $G$ in ($\ref{G}$) can be represented as the following form, which corresponds to setting $\varsigma=d_1+d_2$ and $\vartheta=0$ in \eqref{eq:Xi} and in this case the condition \pref{eq:d1pos} is equivalent to \pref{eq:dalbt}:
\begin{align}\label{eq:gun}
    G= G(u) =
    \begin{pmatrix}
      z_1'& \xi^\mathsf T \\
       \xi & \sqrt{-4D_N''(0)} (\GOE_{N-1}-z_3'\mathbf{I}_{N-1})
    \end{pmatrix}=: \begin{pmatrix}
        z_1'& \xi^\mathsf T \\
         \xi & G_{**}
      \end{pmatrix}  ,
  \end{align}
  where with $z_1,z_2,z_3$ being independent standard Gaussian random variables,
  \begin{align*}
    z_1'&=\si_1 z_1 - \si_2  z_2 + m_1, \quad
z_3'=\frac1{\sqrt{-4D_N''(0)}}\Big(\si_2 z_2+   \sqrt{\al\beta}\rho z_3 - m_2\Big),
  \end{align*}
  $\xi$ is a centered column Gaussian vector with covariance matrix $-2D_N''(0)\mathbf{I}_{N-1}$, which is independent of $z_1,z_2,z_3$ and the GOE matrix $\GOE_{N-1}$. The conditional
distribution of $z_1^{\prime}$ given $z_3^{\prime}=y$ is
\begin{align}\label{eq:cond}
( z_1^{\prime} \mid z_3^{\prime}=y ) \sim N\left(\overline{\mathrm{a}}, \mathrm{b}^2\right),
\end{align}
where
\begin{align}\label{eq:ab2}
\overline{\mathrm{a}}&=  m_1-\frac{\sigma_2^2\left(\sqrt{-4 D_N^{\prime \prime}(0)} y+m_2\right)}{\sigma_2^2+\alpha \beta \rho^2}\nonumber \\
&=  \frac{-2 D_N^{\prime \prime}(0) \alpha \rho^2 u}{\left(-2 D_N^{\prime \prime}(0)-\beta^2\right) \sqrt{D_N\left(\rho^2\right)-\frac{D_N^{\prime}\left(\rho^2\right)^2 \rho^2}{D_N^{\prime}(0)}}} -\frac{\left(-2 D_N^{\prime \prime}(0)-\beta^2-\alpha \beta \rho^2\right) \sqrt{-4 D_N^{\prime \prime}(0)} y}{-2 D_N^{\prime \prime}(0)-\beta^2}, \nonumber \\
\mathrm{b}^2&=  \sigma_1^2+\sigma_2^2-\frac{\sigma_2^4}{\sigma_2^2+\alpha \beta \rho^2}=-4 D_N^{\prime \prime}(0)+\frac{2 D_N^{\prime \prime}(0) \alpha^2 \rho^4}{-2 D_N^{\prime \prime}(0)-\beta^2} .
\end{align}
Define the random variable
\begin{align}\label{aN}
\mathrm{a}_N
=\mathrm{a}_N(\rho, u, y)
=\overline{\mathrm{a}} - \sqrt{-D_N''(0)} \sum_{i=1}^{N-1} \frac{Z_i^2}{ \lambda_i-y},
\end{align}
where $Z_i, 1\leq i \leq N-1,$ are independent standard Gaussian random variables and $\lambda_1 \leq \cdots \leq \lambda_{N-1}$ are the ordered eigenvalues of the GOE matrix $\GOE_{N-1}$. By the above analysis, given Assumption \ref{assumption3}, we can express the expected number of critical points (with or without given indices) of non-isotropic Gaussian random fields with isotropic increments using the eigenvalue density of the GOE matrix. We omit the proof here since it follows from an easy adaption of the arguments in \cite[Section 4]{AZ20} and \cite[Section 3.1]{AZ22}.

\begin{theorem}\label{shell}
Let $H_N=\left\{H_N(x), x \in \mathbb{R}^N\right\}$ be a non-isotropic Gaussian field with isotropic increments. Let $E \subset \mathbb{R}$ be a Borel set and $T_N$ be the shell domain $T_{N}\left(R_{1}, R_{2}\right)=\{x \in \mathbb{R}^{N}: R_{1}<\|x\|<R_{2}\}$, where $0 \leq R_{1}<R_{2} < \infty$.
Assume Assumptions \ref{assumption1}, \ref{assumption2}, and \ref{assumption3}. Then we have
\begin{align*}
 &\mathrel{\phantom{=}} \mathbb{E} [\operatorname{Crt}_N\left(E, \left(R_{1}, R_{2}\right)\right)]\\
 &=\frac{2\left(-2D_N''(0)\right)^{N/2}}{ D_N^{\prime}(0)^{N / 2}\Gamma(N / 2)}\int_{R_{1}}^{R_{2}} \int_{E} \int_{\mathbb{R}} \frac{\exp \left\{ -\frac{\left(\sqrt{-4 D_N^{\prime \prime}(0)} y+m_2\right)^2}{2\left(-2 D_N^{\prime \prime}(0)-\beta^2\right)} - \frac{u^{2}}{2 \sigma_{Y}^{2}} \right\}}{2 \pi \sigma_{Y} \sqrt{-2 D_N^{\prime \prime}(0)-\beta^2}} \rho^{N-1}\\
& \quad \times \mathbb{E}_Z \left\{\mathbb{E}_{\mathrm{GOE}}\left[ \left( \mathrm{a}_N \left( \Phi\left(\frac{ \mathrm{a}_N}{\mathrm{~b}}\right) -\Phi\left(- \frac{ \mathrm{a}_N}{\mathrm{~b}}\right) \right) + \frac{\sqrt{2}\mathrm{b}}{\sqrt{\pi }} e^{-\frac{\mathrm{a}_N^2}{2 \mathrm{b}^2}}\right) \prod_{j=1}^{N-1}\left|\lambda_j-y\right| \right]\right\} \mathrm{d} y \, \dd  u \, \dd  \rho,
\intertext{for $k = 0, 1, \dots, N$,}
&\mathrel{\phantom{=}} \mathbb{E} [\operatorname{Crt}_{N,k}\left(E, \left(R_{1}, R_{2}\right)\right)]\\
&=\frac{2\left(-2D_N''(0)\right)^{N/2}}{ D_N^{\prime}(0)^{N / 2}\Gamma(N / 2)}\int_{R_{1}}^{R_{2}} \int_{E} \int_{\mathbb{R}} \frac{\exp \left\{ -\frac{\left(\sqrt{-4 D_N^{\prime \prime}(0)} y+m_2\right)^2}{2\left(-2 D_N^{\prime \prime}(0)-\beta^2\right)} - \frac{u^{2}}{2 \sigma_{Y}^{2}} \right\}}{2 \pi \sigma_{Y} \sqrt{-2 D_N^{\prime \prime}(0)-\beta^2}} \rho^{N-1} \\
&\quad \times \mathbb{E}_Z \left\{ \mathbb{E}_{\mathrm{GOE}} \left[ \left( \mathrm{a}_N \Phi\left(\frac{ \mathrm{a}_N}{\mathrm{b}}\right)+\frac{\mathrm{b}}{\sqrt{2 \pi }} e^{-\frac{ \mathrm{a}_N^2}{2\mathrm{b}^2}}\right) \prod_{j=1}^{N-1}\left|\lambda_j-y\right| \mathbf{1} \left\{\lambda_k< y<\lambda_{k+1}, \, 0 \leq k \leq N-1 \right\} \right. \right. \\
&\quad \left. \left. + \left(-\mathrm{a}_N \Phi\left(-\frac{ \mathrm{a}_N}{\mathrm{b}}\right)+\frac{\mathrm{b}}{\sqrt{2 \pi }} e^{-\frac{ \mathrm{a}_N^2}{2\mathrm{b}^2}}\right) \prod_{j=1}^{N-1}\left|\lambda_j-y\right| \mathbf{1} \left\{\lambda_{k-1}<y<\lambda_{k}, \, 1 \leq k \leq N \right\} \right]\right\} \mathrm{d} y \, \dd  u \, \dd  \rho,
\end{align*}
where $m_2$ and $ \si_Y$ are given in \eqref{eq:msialbt}, $\mathrm{b}$ is given in \eqref{eq:ab2}, $\mathrm{a}_N$ is defined with the GOE eigenvalues $\lambda_i$ and the independent standard Gaussian random variables $Z_i$, $1\leq i \leq N-1$ in \eqref{aN}, $\Phi$ is the c.d.f.\ of the standard Gaussian random variable and by convention $\lambda_0=-\infty$, $\lambda_N=\infty$.
\end{theorem}

\section{Representation for fields with structure functions in $\mathcal{D}_\infty$} \label{sec3}

In this section, we show that if a structure function has the representation \eqref{Dr}, then Assumption \ref{assumption3} always holds.
\begin{lemma}\label{le:assiii12}
When $D (r)$  has the form \eqref{Dr}, if the inequality \eqref{ar} holds, then both inequalities
\eqref{larn} and \eqref{be} hold automatically.
\end{lemma}
\begin{proof}
Since $D (r)$ has the form \eqref{Dr}, the function $D^{\prime}(r)$ is positive and strictly convex on $[0,\infty)$. Together with the mean value theorem, for any $r>0$, $D^{\prime\prime}(r)r>D^{\prime}(r)-D^{\prime}(0)$, which further implies
\begin{align}\label{ab}
\alpha r>2 \beta,\quad  \quad\alpha\beta r<2 \beta^2.
\end{align}
If the inequality \eqref{ar} holds, by \eqref{ab}, we deduce that $\left(\alpha r+\beta\right) \alpha r<2\left(\alpha r+\beta\right)\beta<-4 D^{\prime \prime}(0)$, which yields the inequality \eqref{be}. On the other hand, for any $r>0$, note that
\begin{align*}
\frac{2D^{\prime\prime}(r)^2
r^{2}+\left(D^{\prime}(r)-D^{\prime}(0)\right)^2}{D\left(r\right)-\frac{D^{\prime}(r)^2 r}{D^{\prime}(0)}}=\frac{\alpha^2 r^2}{2}+\beta^2.
\end{align*}
Thus, to prove \eqref{larn}, we only need to show
\begin{align*}
-2D^{\prime\prime}(0)>\frac{\alpha^2 r^2}{2}+\beta^2,
\end{align*}
which follows from \eqref{ab} by observing that $\frac{\alpha^2 r^2}{2}+\beta^2<\left(\alpha r+\beta\right) \beta<-2D^{\prime\prime}(0)$. 
\end{proof}
Recall that
a function $f \colon (0, \infty) \to \mathbb{R}$ is a Bernstein function if $f$ is of class $C^{\infty}, f(\lambda) \geq 0$ for all $\lambda>0$ and
$(-1)^{n-1} f^{(n)}(\lambda) \geq 0$ for all $n \in \mathbb{N}$ and $\lambda>0$. For the Bernstein functions, we have the following theorem.
\begin{theorem}[{\cite[Theorem 3.2]{SSV12}}]
A function $f \colon (0, \infty) \to \mathbb{R}$ is a Bernstein function if and only if, it admits the representation
$$
f(\lambda)=a+b \lambda+\int_{(0, \infty)}\left(1-e^{-\lambda t}\right) \mu(\dd t),
$$
where $a, b \geq 0$ and $\mu$ is a measure on $(0, \infty)$ satisfying $\int_{(0, \infty)}(1 \wedge t) \mu(\dd t)<\infty$. In particular, the triplet $(a, b, \mu)$ determines $f$ uniquely and vice versa.\label{Bern}
\end{theorem}
In the paper \cite{AZ20}, the authors conjectured that all Bernstein functions with the form of $D(r)$ defined in \eqref{Dr} satisfy Assumption \ref{assumption3}. The following result shows that it is indeed true. Thanks to \pref{le:assiii12}, it remains to prove the inequality \eqref{ar}.
\begin{theorem} \label{thm3.3}
Assume Assumptions \ref{assumption1} and \ref{assumption2}. Let $D (r)$ denote the Bernstein function defined in \eqref{Dr}. Then we have for all $r>0$,
\begin{equation*}
    -2 D'' (0) > (\alpha r + \beta) \beta,
\end{equation*}
where $\alpha=\alpha\left(r \right)$ and $\beta=\beta\left(r \right)$ are defined in \eqref{eq:msialbt} with $\rho^2$ replaced by $r$.
\end{theorem}

\begin{proof}
    It suffices to show
    \begin{align}
        -2 D''(0) \left( D(r) + r D'(0) - 2r D'(r) \right) &> 2r D''(0) \left( D'(r) - D'(0) \right) + \left( D'(r) - D'(0) \right)^2, \label{eq:thm2pr1} \\
        2r D''(0) \frac{\left( D'(r) - D'(0) \right)^2}{D'(0)} &\geq 2r \left( D''(r) - D''(0) \right) \left( D'(r) - D'(0) \right), \label{eq:thm2pr2}
    \end{align}
    for any $r > 0$. Indeed, adding the above two inequalities together implies that for all $r > 0$,
    \begin{equation*}
        -2 D'' (0) \left( D(r) - \frac{D' (r)^2}{D' (0)} r \right)
        > 2r D'' (r) \left( D'(r) - D'(0) \right) + \left( D'(r) - D'(0) \right)^2, 
    \end{equation*}
    which is equivalent to the desired conclusion.

    To prove \eqref{eq:thm2pr1}, notice that $D'' (0) < 0$ and
    \begin{gather*}
        \left( D'(r) - D'(0) \right)^2
        = \left[ \int_{(0, \infty)} t \left( e^{-rt} - 1 \right) \nu(\mathrm{d} t) \right]^2
        \leq -D''(0) \int_{(0, \infty)} \left( e^{-rt} - 1 \right)^2 \nu(\mathrm{d} t).
    \end{gather*}
    Therefore, it suffices to show that for $r>0$,
    \begin{equation*}
        2 \left( D(r) + r D'(0) - 2r D'(r) \right) > -2r \left( D'(r) - D'(0) \right) + \int_{(0, \infty)} \left( e^{-rt} - 1 \right)^2 \nu(\mathrm{d} t), 
    \end{equation*}
    or equivalently, $2D(r) - 2r D'(r) - \int_{(0, \infty)} ( 1 - e^{-rt} )^2 \nu(\mathrm{d} t) > 0$. But from the definition of $D(r)$, we have for $r>0$,
    \begin{equation*}
        2D(r) - 2r D'(r) - \int_{(0, \infty)} \left( 1 - e^{-rt} \right)^2 \nu(\mathrm{d} t)
        = \int_{(0, \infty)} \left( 1 - 2rt e^{-rt} - e^{-2rt} \right) \nu(\mathrm{d} t)
        > 0, 
    \end{equation*}
    where in the last inequality we used the fact that $1 - 2x e^{-x} - e^{-2x} > 0$ for any $x > 0$. 

    To prove \eqref{eq:thm2pr2}, it suffices to show that for $r>0$,
    \begin{equation*}
        D''(0) \left( D'(r) - D'(0) \right) \leq D'(0) \left( D''(r) - D''(0) \right), 
    \end{equation*}
    as $D'(0) > 0$ and $D'(r) - D'(0) < 0$. By definition,
    \begin{align*}
        &\mathrel{\phantom{=}} D''(0) \left( D'(r) - D'(0) \right) - D'(0) \left( D''(r) - D''(0) \right) \\
        &= -\int_{(0, \infty)} s^2 \, \nu (\mathrm{d} s) \int_{(0, \infty)} t \left( e^{-rt} - 1 \right) \nu (\mathrm{d} t) - \left( \int_{(0, \infty)} t \, \nu (\mathrm{d} t) + A \right) \int_{(0, \infty)} s^2 \left( 1 - e^{-rs} \right) \nu (\mathrm{d} s) \\
        &= \int_{(0, \infty)^2} s^2 t \left( e^{-rs} - e^{-rt} \right) \nu (\mathrm{d} s) \nu (\mathrm{d} t) - A \int_{(0, \infty)} s^2 \left( 1 - e^{-rs} \right) \nu (\mathrm{d} s) \\
        &= \frac{1}{2} \int_{(0, \infty)^2} s t (s - t) \left( e^{-rs} - e^{-rt} \right) \nu (\mathrm{d} s) \nu (\mathrm{d} t) - A \int_{(0, \infty)} s^2 \left( 1 - e^{-rs} \right) \nu (\mathrm{d} s).
    \end{align*}
    Since $A \geq 0$, $s^2 \left( 1 - e^{-rs} \right) > 0$, and $s t (s - t) \left( e^{-rs} - e^{-rt} \right) \leq 0$ for any $s, t, r > 0$, we find for $r>0$,
    \begin{equation*}
        D''(0) \left( D'(r) - D'(0) \right) - D'(0) \left( D''(r) - D''(0) \right)
        \leq 0,
    \end{equation*}
    which gives \eqref{eq:thm2pr2} and completes the proof.
\end{proof}

On the other hand, Assumption \ref{assumption3} is not expected to hold for all structure functions in the class $\dx_N$ with fixed $N\in\nz$, since the Bessel functions in the representation \eqref{DNr} are oscillatory. The following example gives a family of functions in $\dx_1 \setminus \dx_\8$ satisfying Assumption \ref{assumption3}.

\begin{example} \label{ex2} \rm
For any $\varepsilon \in (0, \frac{1}{8}]$, the function
\begin{align*}
    f (r) = (r + 1)^{\frac{11}{12}} - 1 + \varepsilon \int_0^r \frac{1 - \cos \sqrt{t}}{t} \, \mathrm{d} t
\end{align*}
is in $\mathcal{D}_1 \setminus \mathcal{D}_\infty$ and satisfies Assumption \ref{assumption3}. For simplicity, we denote
\begin{align*}
    f_1 (r) = (r + 1)^{\frac{11}{12}} - 1, \qquad \qquad
    f_2 (r) = \int_0^r \frac{1 - \cos \sqrt{t}}{t} \, \mathrm{d} t.
\end{align*}
Then $f (r) = f_1 (r) + \varepsilon f_2 (r)$ for $\varepsilon \in (0, \frac{1}{8}]$. Let us verify that $f$ satisfies the desired properties. 

(1) We first show that $f \in \mathcal{D}_1 \setminus \mathcal{D}_\infty$. Indeed, it is well known that $f_1$ is a Bernstein function \cite{SSV12} with $f_1 (0) = 0$, which yields $f_1 \in \mathcal{D}_\infty \subset \mathcal{D}_1$. For $N=1$, we have $\Lambda_1(x) = \cos x$. If we set $\nu_1 (\mathrm{d} t) = \frac{1}{t} \mathbf{1}_{(0,1)} (t) \, \mathrm{d} t$, then
\begin{align*}
    f_2 (r)
    = \int_0^r \frac{1 - \cos \sqrt{t}}{t} \, \mathrm{d} t
    = \int_0^1 \frac{1 - \cos \sqrt{rt}}{t} \, \mathrm{d} t
    = \int_{(0, \infty)} \left( 1 - \Lambda_1 (\sqrt{rt}) \right) \nu_1 (\mathrm{d} t).
\end{align*}
It follows that $f_2 \in \mathcal{D}_1$ and thus $f = f_1 + \varepsilon f_2 \in \mathcal{D}_1$. To check $f \notin \mathcal{D}_\infty$, we compute
\begin{align*}
    f_1' (r) &= \frac{11}{12} (r+1)^{-\frac{1}{12}}, &
    f_1'' (r) &= -\frac{11}{144} (r+1)^{-\frac{13}{12}}, \\
    f_2' (r) &= \begin{cases}
        \frac{1 - \cos \sqrt{r}}{r}, &r > 0, \\
        \frac{1}{2}, &r = 0,
    \end{cases}  &
    f_2'' (r) &= \begin{cases}
        \frac{\sqrt{r} \sin \sqrt{r} + 2 \cos \sqrt{r} - 2}{2 r^2}, &r > 0, \\
        -\frac{1}{24}, &r = 0,
    \end{cases}
\end{align*}
and
\begin{align*}
    f_1''' (r) = o \left( \frac{1}{r^2} \right), \qquad
    f_2''' (r) = \frac{\cos \sqrt{r}}{4 r^2} + o \left( \frac{1}{r^2} \right), \qquad \text{as } r \to \infty.
\end{align*}
Therefore, $f''' (r) = f_1''' (r) + \varepsilon f_2''' (r) = (\varepsilon \cos \sqrt{r} + o (1)) / (4 r^2)$, as $r \to \infty$. From here we can find an $r_0 > 0$ such that $f''' (r_0) < 0$, which indicates $f \notin \mathcal{D}_\infty$.

(2) We check $\alpha, \beta < 0$. This would follow from $f'' (r) < 0$ for all $r>0$. In fact, from calculus we know $x \sin x + 2 \cos x - 2 < 0$ for $x \in (0, \pi)$, which implies $f_2'' (r) < 0$ for all $r \in (0, 9)$. Together with $f_1'' (r) < 0$, we have $f'' (r) < 0$ for $r \in (0, 9)$. And for $r \in [9, + \infty)$, since $f_2'' (r) \leq (\sqrt{r} \sin \sqrt{r}) / (2 r^2) \leq 1 / (2 r^{3 / 2})$, we obtain
\begin{align*}
    f'' (r)
    = f_1'' (r) + \varepsilon f_2'' (r)
    \leq - \frac{11}{144} (r+1)^{-\frac{13}{12}} + \frac{\varepsilon}{2} r^{-\frac{3}{2}}
    \leq \left[ -\frac{11}{144} \left( \frac{r}{r+1} \right)^{\frac{3}{2}} + \frac{1}{16} \right] r^{-\frac{3}{2}}
    < 0.
\end{align*}

(3) With the same decomposition as in the proof of Theorem \ref{thm3.3}, \eqref{ar} follows from
\begin{align}
    -2 f''(0) \left( f(r) + r f'(0) - 2r f'(r) \right) &> 2r f''(0) \left( f'(r) - f'(0) \right) + \left( f'(r) - f'(0) \right)^2, \label{eq1} \\
    2r f''(0) \frac{\left( f'(r) - f'(0) \right)^2}{f'(0)} &\geq 2r \left( f''(r) - f''(0) \right) \left( f'(r) - f'(0) \right). \label{eq2}
\end{align}
Notice that \eqref{eq1} is equivalent to
$g (r) := ( f' (r) - f' (0) )^2 + 2 f'' (0) ( f(r) - r f' (r)) < 0$. One can check that $x \sin x + 2 \cos x - 2 > - \frac{x^4}{12}$ for $x > 0$, which further implies $f_2'' (r) > -\frac{1}{24} = f_2'' (0)$ for $r > 0$. Since $f_1'' (r) > f_1'' (0)$, we have $f'' (r) > f'' (0)$ for $r > 0$. Therefore,
\begin{align*}
    g' (r)
    = 2 f'' (r) \left( f' (r) - f' (0) - r f'' (0) \right)
    = 2 f'' (r) \int_0^r (f'' (s) - f'' (0) ) \, \mathrm{d} s
    < 0 \ \  \text{for} \ \ r > 0,
\end{align*}
which implies $g(r) < 0$ and \eqref{eq1} for $r > 0$.

Since $f' (r) - f' (0) < 0$, \eqref{eq2} is equivalent to $f'' (0) ( f' (r) - f' (0) ) \leq f' (0) ( f'' (r) - f'' (0) )$. Since $f' (0) = -12 f'' (0)$, it suffices to prove $f' (r) + 12 f'' (r) \geq 0$. By Taylor expansion we know $1 - \cos x > \frac{x^2}{2} - \frac{x^4}{24}$ for $x > 0$. It follows that $f_2' (r) > \frac{1}{2} - \frac{r}{24}$ for $r > 0$. Together with $f_2'' (r) > -\frac{1}{24}$, we have
\begin{align*}
    f_2' (r) + 12 f_2'' (r) > - \frac{1}{24} r 
\end{align*}
for $ r > 0.$ If $r \in (0, 9)$, we have
\begin{align*}
    f_1' (r) + 12 f_1'' (r)
    = \frac{11}{12} (r+1)^{-\frac{13}{12}} r
    > \frac{11}{1200} r,
\end{align*}
which implies that for $r\in (0,9)$,
\begin{align} \label{eq3}
    f' (r) + 12 f'' (r)
    > \left( \frac{11}{1200} - \frac{\varepsilon}{24} \right) r
    > 0.
\end{align}
For $r \in [9, +\infty)$, since $f_2'(r)\ge0$, we find
\begin{gather*}
    f_1' (r) + 12 f_1'' (r)
    = \frac{11}{12} (r+1)^{-\frac{13}{12}} r
    \geq \frac{11}{12}  \left( \frac{9}{10} \right)^{\frac{13}{12}} r^{-\frac{1}{12}}, \\
    f_2' (r) + 12 f_2'' (r)
    \geq -12 \times \frac{\left\lvert \sqrt{r} \sin \sqrt{r} + 2 \cos \sqrt{r} - 2 \right\rvert}{2 r^2}
    \geq -30 r^{-\frac{3}{2}},
\end{gather*}
which implies that
\begin{align} \label{eq4}
    f' (r) + 12 f'' (r)
    \geq \left[ \frac{11}{12}  \left( \frac{9}{10} \right)^{\frac{13}{12}} - 30 \varepsilon r^{-1} \right] r^{-\frac{1}{12}}
    > 0
\end{align}
for $r \in [9, +\infty)$. Now \eqref{eq2} follows from \eqref{eq3} and \eqref{eq4}.

 (4) It remains to show \eqref{be}. As in \pref{le:assiii12}, we just have to prove $\alpha r > 2 \beta$.  To this end, we consider
\begin{align*}
    r f_2'' (r) - ( f_2'(r) - f_2' (0) )
    = \frac{r + \sqrt{r} \sin \sqrt{r} + 4 \cos \sqrt{r} - 4}{2r}.
\end{align*}
One can check that $r + \sqrt{r} \sin \sqrt{r} + 4 \cos \sqrt{r} - 4>0$ for $r > 0$. Note that $r f_1'' (r) - ( f_1'(r) - f_1' (0) ) > 0$. It follows that $r f'' (r) - ( f'(r) - f' (0) ) > 0$ for $r > 0$ which further implies $\alpha r > 2 \beta$.

\end{example}

\begin{remark}
    From Example \ref{ex2}, we also find that $f_2\in \dx_1\setminus\dx_\8$. Since $f_2'$ is not monotone, $f_2$ clearly violates Assumption \ref{assumption3}.

    In the isotropic case, it was shown in \cite[Section 3.3]{CS18} that one can use GOE in the Kac--Rice representation if and only if $\kappa^2:=\frac{B_N'(0)^2}{B_N''(0)}<1$. In particular, we have $\kappa^2<1$ if $B(r^2)$ is positive definite on $\rz^N$ for all $N\in\nz$. Consider $B(r)=1-p+p\Lambda_2(\sqrt{r})$ for $0<p<\frac12$. One can check directly $\kappa^2<1$ but $B(r)$ is not completely monotone. It follows from Schoenberg's theorem \cite{Sch38} that $B(r^2)$ is not positive definite on $\rz^d$ for $d$ large enough. 

    The above discussion shows that for locally isotropic Gaussian random fields the class $\dx_\8$ is sufficient for the GOE representation but not necessary.
\end{remark}

\begin{acknowledgements}
We are grateful to the anonymous referees for careful reading and many constructive suggestions which have significantly improved the paper. 
\end{acknowledgements}

\bibliographystyle{imsart-number}
\bibliography{reference}
\end{document}